\documentclass{amsart}
\usepackage{amsfonts}
\usepackage{amssymb}
\usepackage{graphicx}
\usepackage{enumerate}
\usepackage{amsmath}
\usepackage{relsize} 

\newtheorem{theorem}{Theorem}[section]
\newtheorem{lemma}[theorem]{Lemma}
\newtheorem{corollary}[theorem]{Corollary}

\theoremstyle{definition}

\theoremstyle{remark}

\numberwithin{equation}{section}

\newcommand{\refth}[1]{Theorem~\ref{#1}}
\newcommand{\reflm}[1]{Lemma~\ref{#1}}
\newcommand{\refco}[1]{Corollary~\ref{#1}}

\newcommand{\refeq}[1]{(\ref{#1})}

\newcommand{\beeq}[1]{\begin{equation} \label{#1}}
\newcommand{\eeq}{\end{equation}}
\renewcommand{\(}{\begin{eqnarray*}}
\renewcommand{\)}{\end{eqnarray*}}
\newcommand{\beeqn}{\begin{eqnarray}}
\newcommand{\eeqn}{\end{eqnarray}}
\newcommand{\bq}{\begin{eqnarray}}
\newcommand{\eq}{\end{eqnarray}}
\newcommand{\bs}[1]{{\boldsymbol #1}}
\newcommand{\sothat}{\\ \Rightarrow \hspace*{2mm} &&}
\newcommand{\nexteqline}{\\ &=&}
\newcommand{\eqspace}{\hspace*{5mm}}
\newcommand{\eqand}{\mbox{\hspace*{3mm} and \hspace*{3mm}}}
\newcommand{\eqwhere}{\mbox{\hspace*{3mm} where \hspace*{3mm}}}

\renewcommand{\quad}{\hspace*{5mm}}
\renewcommand{\qquad}{\hspace*{10mm}}

\newcommand{\lp}{\left(  }
\newcommand{\rp}{\right) }
\newcommand{\lb}{\left\{  }
\newcommand{\rb}{\right\} }
\newcommand{\lbr}{\left[  }
\newcommand{\rbr}{\right] }
\newcommand{\lf}{\left\lfloor}
\newcommand{\rf}{\right\rfloor}
\newcommand{\lc}{\left\lceil}
\newcommand{\rc}{\right\rceil}

\newcommand{\Z}{{\mathbb Z}}

\newcommand{\R}{{\mathbb R}}

\newcommand{\ep}{\epsilon}
\newcommand{\FF}{{\mathcal F}}
\newcommand{\GG}{{\mathcal G}}
\newcommand{\DD}{{\mathcal D}}
\newcommand{\YY}{{\mathcal Y}}

\newcommand{\TT}{{\mathcal T}}
\renewcommand{\AA}{{\mathcal A}}

\newcommand{\HH}{{\mathcal H}}

\newcommand{\PP}{{\mathcal P}}

\newcommand{\CC}{{\mathcal C}}

\newcommand{\XX}{{\mathcal X}}

\newcommand{\UU}{{\mathcal U}}

\newcommand{\be}[1]{\begin{enumerate} [#1]}
\newcommand{\ee}{\end{enumerate}}

\newcounter{cnt1}
\newcounter{cnt2}
\newcounter{cnt3}
\newcounter{cnt4}

\newcommand{\bnum}
{
\begin{list}{\arabic{cnt1})}
{
\usecounter{cnt1}
\leftmargin 5mm
\setlength{\leftmargin}{\leftmargin}
\topsep 2pt
\parsep 1pt
\itemsep 1pt}
}
\newcommand{\enum}{\end{list}}

\newcommand{\broman}
{
\begin{list}{\roman{cnt2})}
{
\usecounter{cnt2}
\leftmargin 2mm
\setlength{\leftmargin}{\leftmargin}
\topsep 2pt
\parsep 1pt
\itemsep 1pt}
}
\newcommand{\eroman}{\end{list}}

\newcommand{\bRoman}
{
\begin{list}{\Roman{cnt3})}
{
\usecounter{cnt3}
\leftmargin 5mm
\setlength{\leftmargin}{\leftmargin}
\topsep 2pt
\parsep 1pt
\itemsep 1pt}
}
\newcommand{\eRoman}{\end{list}}

\newcommand{\balph}
{
\begin{list}{\alph{cnt4})}
{
\usecounter{cnt4}
\leftmargin 2mm
\setlength{\leftmargin}{\leftmargin}
\topsep 2pt
\parsep 1pt
\itemsep 1pt}
}
\newcommand{\ealph}{\end{list}}

\newcommand{\bAlph}
{
\begin{list}{\Alph{cnt1})}
{
\usecounter{cnt1}
\leftmargin 5mm
\setlength{\leftmargin}{\leftmargin}
\topsep 2pt
\parsep 1pt
\itemsep 1pt}
}
\newcommand{\eAlph}{\end{list}}

\newcommand{\bdot}
{
\begin{list}{$\cdot$}
{
\leftmargin  3mm
\setlength{\leftmargin}{\leftmargin}
\topsep 3pt
\parsep 1pt
\itemsep 2pt}
}
\newcommand{\edot}{\end{list}}

\newcommand{\bdash}
{
\begin{list}{-}
{
\leftmargin 3mm
\setlength{\leftmargin}{\leftmargin}
\topsep 2pt
\parsep 1pt
\itemsep 1pt}
}
\newcommand{\edash}{\end{list}}

\begin{document}

\title{Improved Bound on Sets Including No Sunflower with Three Petals}


\author{Junichiro Fukuyama}
\address{Department of Computer Science and Engineering\\
The Pennsylvania State University\\
PA 16802, USA}
\curraddr{}
\email{jxf140@psu.edu}
\thanks{}




\date{}

\dedicatory{}

\begin{abstract}
A {\em sunflower with $k$ petals}, or {\em $k$-sunflower}, is a family of $k$ sets every two of which have a common intersection. Known since 1960, the {\em sunflower conjecture} states that a family $\FF$ of sets each of cardinality $m$ includes a $k$-sunflower if $|\FF| \ge c_k^m$ for some $c_k \in \R_{>0}$ depending only on $k$. The case $k=3$ of the conjecture was especially emphasized by Erd\"os, for which Kostochka's bound $c m! \lp \frac{\log \log \log m}{\log \log m} \rp^m$ on $|\FF|$ without a 3-sunflower had been the best-known since 1997 until the recent development to update it to $c \log m$. This paper proves with an entirely different combinatorial approach  that 
$\FF$ includes three mutually disjoint sets if it satisfies the $\Gamma \lp c m^{\frac{1}{2}+ \delta} \rp$-condition for any given $\delta \in (0, 1/2)$. 
Here $c$ is a constant depending only on $\delta$, and the $\Gamma$-condition refers to 
\[
| \lb U~:~ U \in \FF \textrm{~and~} S \subset U \rb|
< \lp c m^{\frac{1}{2}+ \delta} \rp^{-|S|} |\FF|, 
\]
for every nonempty set $S$. 
This poses an alternative proof of the 3-sunflower bound $\lp c m^{\frac{1}{2}+ \delta} \rp^m$. 
\end{abstract}

\maketitle

\section{Motivation and Approach}
In this paper we verify the following statement.

\begin{theorem} \label{3SF} 
For each $\delta \in (0, 1/2)$, there exists $c \in \R_{>0}$ such that a family $\FF$ of sets each of cardinality $m \in \Z_{>0}$ includes three mutually disjoint sets if it satisfies the 
$
\Gamma \lp c m^{\frac{1}{2}+ \delta} \rp
$-condition. \qed
\end{theorem}

This means that $\FF$ includes a 3-sunflower if $|\FF|>\lp c m^{\frac{1}{2}+ \delta} \rp^m$, since for such an $\FF$, there exists a set $S$ with $|S|<m$ such that the family $\lb U - S~:~ U \in \FF \textrm{~and~} S \subset U \rb$ satisfies the $
\Gamma \lp c m^{\frac{1}{2}+ \delta} \rp
$-condition in the universal set minus $S$. The claim asymptotically updates Kostochka's bound \cite{Ko97, ASU13} that had been the best-known related to the three-petal sunflower problem noted in \cite{E81} since 1997, until the recent development \cite{ALWZ20} to reduce the upper bound to $c \log m$. 

We prove the statement with a new combinatorial theory describing its approach in the rest of the section.

\subsection{$l$-Extension of a Family of $m$-Sets} 
Let the universal set $X$ have cardinality $n$. Denote a subset of $X$ by a capital alphabetical letter. It is an {\em $m$-set} if its cardinality is $m$. Use the standard notation $[p] := \lbr 1, p \rbr \cap \Z$ for $p \in \Z$, and ${X'\choose m} := \lb U ~:~ U \subset X',~|U| =m  \rb$ for $X' \subset X$. For $\FF \subset {X \choose m}$, we denote 
\(
\FF[S]  := \lb U ~:~ U \in \FF,~ S \subset U  \rb.
\)
The family $\FF$ satisfies {\em the $\Gamma (b)$-condition} $\lp b \in \R_{>0}  \rp$ if 
$
|\FF[S]| < b^{-|S|} |\FF| 
$ for every nonempty set $S$.

The {\em $l$-extension of $\FF$} for $l \in [n] - [m]$ is defined as
\[
Ext\lp \FF, l \rp:= \lb T ~:~ T \in {X \choose l}, \textrm{~and~}
\exists U \in \FF, U \subset T
\rb.
\]
It is shown in \cite{my6} that
\beeq{eqMy6}
\left| Ext \lp \FF, l \rp \right|  \ge {n \choose l} \lb 1 - m
\exp \lbr \frac{-(l-m+1)|\FF|}{8 m! {n \choose m}} \rbr \rb,
\eeq
for any $\FF \subset {X \choose m}$ and $l \in [n] - [m]$. The result means that an $n$-vertex graph $G$ with ${n \choose 2} - k$ edges contains at most $\lf 2 {n \choose l} \exp \lbr - \frac{(l-1)k}{8n(n-1)}   \rbr \rf$ cliques of size $l$:
let the vertex set of $G$ be $X$, and $\FF$ be the set of non-edges in $G$ regarded as a family of 2-sets. Then $Ext \lp \FF, l \rp$ equals the family of $l$-sets each not a clique of size $l$ in $G$, which means the claim.  Similar facts can be seen for $m$-uniform hypergraphs for small $m$ such as 3.

\subsection{Existence of a Bounded Set $T$ with Dense $Ext\lp  \FF[T] , l  \rp$}

An {\em $(l, \lambda)$-extension generator of $\FF$} is a set $T \subset X$ such that
\[
\left| Ext\lp  \FF[T] , l  \rp \right| \ge {n-|T| \choose l-|T|} \lp 1 - e^{-\lambda} \rp,
\]
where $\lambda \in \R_{>0}$, and $e=2.71...$ is the natural logarithm base.  If $\lambda$ is much larger than a constant, the $l$-sets in $Ext \lp \FF[T], l \rp$ form a vast majority of ${X \choose l} [T]$, the family of $l$-sets each containing $T$.

We have a fact shown in \cite{monotone}.

\begin{theorem}
\label{EGT}
(Extension Generator Theorem)
There exists $\ep \in (0, 1)$ satisfying the following statement:
let $X$ be the universal set of cardinality $n$, $m \in [n-1]$, $l \in [n]-[m]$,  and
$
\lambda \in \lp 1, \frac{\ep l}{m^2} \rp.
$
For every nonempty family $\FF \subset {X \choose m}$, there exists an $\lp l, \lambda \rp$-extension generator $T$ of $\FF$ with  
$
|T| \le  \lbr \ln {n \choose m} - \ln |\FF| \rbr \Bigr/ \ln \frac{\ep l}{m^2 \lambda}.
$  \qed
\end{theorem}

We will also confirm it in Section 2. The theorem could help us understand the structure of $Ext \lp \FF, l \rp$: for some large family $\FF$, we can find bounded sets \\$T_1, T_2, \ldots, T_k$ such that $Ext \lp \FF, l \rp$ is close to $\bigcup_{i \in [k]} {X \choose l} \lbr T_i \rbr$.

In addition, an alternative proof has been given \cite{monotone} with the theorem that the monotone complexity of detecting cliques in an $n$-vertex graph is exponential. For any given polynomial-sized monotone circuit $\CC$ for the $k$-clique problem ($k=n^\ep$ for some constant $\ep\in (0, 1)$), the proof explicitly constructs a graph containing no $k$-clique for which $\CC$ returns true.
The standard method to show the exponential complexity uses the sunflower lemma or its variant with random vertex coloring \cite{Raz, AB87}.

\subsection{To Show \refth{3SF}}

Given $\FF \subset {X \choose m}$, we first parition $X$ into equal sized disjoint sets $X_1, X_2, \ldots, X_r$ ($r \simeq m^{\frac{1 - \delta}{2}}$) such that $ \frac{m}{2r} < |U \cap X_j| < \frac{2m}{r}$ for every $j \in [r]$ and most $U \in \FF$. Find such $X_j$ by the claims we show in Section 3. Then we will inductively construct three families $\FF_i$ ($i \in [3]$) of $U$ for each $j$ such that $U_i \cap \bigcup_{j' \in [j]} X_{j'}$ are mutually disjoint for any three $U_i \in \FF_i$. The recursive invariant is verified by claims closely related to \refth{EGT}, which we will prove in the following section.

\section{Proof of \refth{EGT} And Other Facts}

\subsection{A Structural Lemma} 
Denote 
$\underbrace{\UU \times \UU \times \cdots \times \UU}_g$ by $\UU^g$ for $\UU \subset 2^X$ and $g \in \Z_{>0}$,  also writing 
\[
union \lp \bs{U} \rp = \bigcup_{p=1}^g U_p, \quad 
\textrm{for~} \bs{U} = \lp U_1, U_2, \ldots, U_g \rp  \in \UU^g. 
\]
Let $w: \lp 2^X \rp^g \rightarrow \R_{\ge 0}$ and $m \in [n]$ be given in addition to $g$. These define the {\em norm $\left\| \UU \right\|$ of $\UU$} 
and {\em sparsity $\kappa \lp \FF \rp$} of $\FF \subset {X \choose m}$ by 
\[
\| \UU \|  = \lbr \sum_{\bs{U} \in \UU^g} w \lp \bs{U}  \rp \rbr^{\frac{1}{g}}, 
\eqand 
\kappa \lp \FF \rp = \ln {n \choose m}- \ln \| \FF \|, 
\] 
respectively.

Given such an $\FF$, and numbers $l \in [n]- [m]$ and $j \in \Z_{\ge 0}$, denote 
\( 
&& 
\PP_{j, g} = \lb \bs{U}~:~ \bs{U} \in \FF^g,~ 
\left| union\lp \bs{U} \rp  \right| = gm-j 
\rb, 
\\ &&
\DD_g = \lb (\bs{U}, Y) ~:~ \bs{U} \in \FF^g,~ 
Y \in {X \choose l},~ union(\bs{U}) \subset Y 
\rb, 
\\ &&
\| \PP_{j, g} \| = \sum_{\bs{U}\in \PP_{j, g}} w  \lp \bs{U} \rp,
\eqand 
\| \DD_g \| = \sum_{(\bs{U}, Y) \in \DD_g} w  \lp \bs{U} \rp, 
\)
extending the norm $\| \cdot \|$ for $\PP_{j, g}$ and $\DD_g$, for which we say $w$ {\em induces} $\| \cdot \|$ and also the sparsity $\kappa$. 
The family $\FF$ satisfies the {\em $\Gamma_g (b, h)$-condition on $\| \cdot \|$} $\lp b, h  \in \R_{>0} \rp$  if 
\( 
&& 
\| \UU \|  = \lbr \sum_{\bs{U} \in \lp \UU \cap \FF \rp^g} w \lp \bs{U}  \rp \rbr^{\frac{1}{g}}, \quad 
\textrm{for all~} \UU \subset 2^X, 
\\ \textrm{and} &&
\| \PP_{j, g} \| 
< h b^{-j} 
\| \FF \|^g, \quad 
\textrm{for every~} j \in [(g-1)m]. 
\)
We may drop the subscript $g$ if it is obvious from the context, so $\PP_{j, g}$ can be written as $\PP_j$, $\Gamma_g (b,  h)$ as $\Gamma (b, h)$ etc.

\medskip

In this subsection, we prove the following lemma that is a structural claim we will use to show \refth{3SF}, and to derive \refth{EGT}.

\begin{lemma} \label{TildeGSum}
Let 
\be{i)} 
\item $X$ be the universal set weighted by $w: \lp 2^X \rp^g \rightarrow \R_{\ge 0}$ for some $g \in \Z_{\ge 2}$, 
\item $l \in [n]$, $m \in [l-1]$, and $h, \gamma \in \R_{>0}$, such that 
$\gamma$ and $\frac{l}{gm}$ are both sufficiently large, 
\item and $\FF \subset {X \choose m}$ satisfy the $\Gamma_g  \lp \frac{4 \gamma n}{l},~h \rp$-condition on the norm $\| \cdot \|$ induced by $w$.  
\ee 
Then 
\[ 
\| \DD_g  \| < 
\frac{\lp 1 + \frac{h}{\gamma} \rp {n \choose l}{l \choose m}^g}{{n \choose m}^g} 
\| \FF \|^g 
.
\qed 
\]
\end{lemma} 

\medskip 

Since 
\[
\| \DD_g \| 
= \sum_{(\bs{U}, Y) \in \DD_g} w \lp \bs{U} \rp
=\sum_{Y \in{X \choose l}} 
~\sum_{\bs{U} \in \lbr \FF \cap {Y \choose m} \rbr^g} 
w \lp \bs{U} \rp, 
\]
it means: 

\begin{corollary} \label{EGTTildeG}
For such objects and $\ep \in (0, 1)$, there are $\lc \lp 1 - \ep \rp {n \choose l} \rc$ sets $Y \in {X \choose l}$ such that
\[
\sum_{\bs{U} \in \lbr \FF \cap {Y \choose m} \rbr^g} w \lp \bs{U}  \rp
< 
\frac{\lp 1 + \frac{h}{\gamma} \rp {l \choose m}^g}{\ep {n \choose m}^g}
\sum_{\bs{U} \in \FF^g} w \lp \bs{U}  \rp. 
\qed
\]
\end{corollary}

\medskip

Write 
\[
b = \frac{4 \gamma n}{l}, \eqand 
v(n') = {n' \choose m}^{-g+1}  \hspace*{3mm} \prod_{p=1}^{g-1} {n'-pm \choose m},  
\]
for $n' \in [n] - [g m]$ for the proof.  Observe the five remarks. 
\be{A)} 
\item $\PP_j = \emptyset$ if $j > (g-1)m$, so 
\(
\|  \DD  \|  &=& 
\sum_{Y \in{X \choose l}} ~\sum_{\bs{U} \in {Y \choose m}^g} 
w \lp \bs{U} \rp
= 
\sum_{Y \in {X \choose l}} \left\| {Y \choose m} \right\|^g 
\nexteqline
\sum_{j=0}^{(g-1)m} \| \PP_j \| {n- gm+j \choose l - gm+j}
\\ &<&
h  \sum_{j=1}^{(g-1)m} b^{-j}  {n \choose m}^g e^{- g\kappa\lp \FF \rp}
{n-gm +j \choose l-gm+j}, 
\)  
by the $\Gamma_g \lp b, h\rp$-condition of $\FF$, and $\| \FF \|^g = {n \choose m}^g e^{-g \kappa \lp \FF \rp}$. 

\item  
$
{v(l) \le v(n)} 
$,  
since 
\(
&&
\prod_{i=0}^{m-1} \frac{1-\frac{pm}{l-i}}{1-\frac{pm}{n-i}} \le 1, \quad 
\textrm{for~} p \in [g-1], 
\sothat 
\frac{{l- pm \choose m}{l \choose m}^{-1}}{{n-pm \choose m}{n \choose m}^{-1}} \le 1, 
\quad \Rightarrow \quad 
\frac{v(l)}{v(n)} \le 1. 
\)

\item By the identity ${x \choose z}{x-z \choose y-z}= {x \choose y}{y \choose z}$, 
\( 
&&
{n \choose m}^g {n-gm \choose l-gm}  
= 
\frac{1}{v(n)}  {n-gm \choose l-gm} \prod_{p=0}^{g-1} {n-pm \choose m} 
\nexteqline 
\frac{{l-(g-1)m \choose m}}{v(n)}  {n-(g-1)m \choose l-(g-1)m} 
\prod_{p=0}^{g-2} {n-pm \choose m} 
\nexteqline 
\frac{{l-(g-1)m \choose m} {l-(g-2)m \choose m}}{v(n)}  {n-(g-2)m \choose l-(g-2)m} 
\prod_{p=0}^{g-3} {n-pm \choose m} 
\nexteqline 
\cdots =  
\frac{\prod_{p=0}^{g-1} {l-pm \choose m}}{v(n)} {n \choose l}
\\ &=&
\frac{v(l)}{v(n)} {n \choose l} {l \choose m}^g
\le 
{n \choose l} {l \choose m}^g. 
\)

\item So, 
\( 
\| \PP_0 \| {n-gm \choose l-gm}
&\le&
\| \FF \|^g  {n-gm \choose l-gm} 
= {n \choose m}^g {n-gm \choose l-gm}   e^{- g \kappa \lp \FF \rp} 
\\ &\le & 
{n \choose l} {l \choose m}^g e^{- g \kappa \lp \FF \rp} . 
\)

\item For each $j \in [(g-1)m]$,  
\[
{n-gm +j \choose l-gm+j} 
= {n-gm \choose l-gm} \prod_{i=0}^{j-1} \frac{n-gm+j-i}{l-gm+j-i}  
<
\lp \frac{2n}{l} \rp^{j}
{n-gm \choose l-gm},  
\]
since $l$ and $n$ are both sufficiently larger than $gm$. 
\qed 
\ee

\medskip

It suffices to show by the remarks that 
\beeq{eqTildeG} 
\| \DD \| < \lbr  1  
+ h \sum_{j=1}^{(g-1)m} \lp 2 \gamma \rp^{-j}  
\rbr 
{n \choose l}{l \choose m}^g e^{-g \kappa \lp \FF \rp}, 
\eeq 
as its RHS is less than 
$
\frac{\lp 1 + \frac{h}{\gamma} \rp {n \choose l}{l \choose m}^g}{{n \choose m}^g}
\sum_{\bs{U} \in \FF^g} w \lp \bs{U}  \rp$. We see from A), C) and E) that 
\(
&&
\sum_{j=1}^{(g-1)m} \| \PP_j  \|  {n-gm +j \choose l-gm+j} 
\\ &<& \nonumber 
h  \sum_{j=1}^{(g-1)m} b^{-j}   {n \choose m}^g e^{- g\kappa\lp \FF \rp}
{n-gm +j \choose l-gm+j} 
\\ &<& \nonumber 
h e^{-g \kappa \lp \FF \rp} 
{n \choose m}^g  {n-gm \choose l-gm} 
\sum_{j=1}^{(g-1)m} \lp 2 \gamma  \rp^{-j} 
\\ &\le& \nonumber 
h e^{-g \kappa \lp \FF \rp} 
{n \choose l}{l \choose m}^g 
\hspace*{2mm}
\sum_{j=1}^{(g-1)m} \lp 2 \gamma \rp^{-j}.  
\) 
Also by D), \refeq{eqTildeG} is confirmed completing the proof of \reflm{TildeGSum}.

\subsection{Primitive Weight of $X$ for $g=2$} 
Now let the weight $w: \lp 2^X \rp^2 \rightarrow \R_{\ge 0}$ for $g=2$ be defined by 
\[
(U_1, U_2) \mapsto \tilde{w}(U_1) ~\tilde{w} (U_2),  
\quad 
\textrm{for some~}
\tilde{w}: 2^X \rightarrow \R_{\ge 0}. 
\] 
We say that such an $X$ is {\em primitively weighted} by $w$, and that the weight $w$ is {\em primitive with $\tilde{w}$}. This subsection shows the following statement.

\begin{theorem} \label{EGT42} 
Let $X$ be primitively weighted inducing the norm $\| \cdot \|$. 
For every sufficiently small $\ep \in (0, 1)$, and $\FF \subset {X \choose m}$ satisfying the $\Gamma_2 \lp \frac{4 \gamma n}{l},~1  \rp$-condition on $\| \cdot \|$ for some $l \in [n]$, $m \in [l]$, and 
$
\gamma \in \lbr \ep^{-2},~ l m^{-1} \rbr 
$, there are $\lc {n \choose l} \lp 1-  \ep \rp \rc$ sets $Y \in {X \choose l}$ such that
\[
\lp 1 - \sqrt{\frac{2}{\ep \gamma}} \rp
\frac{{l \choose m}}{{n \choose m}}
\left\| \FF \right\|
<
\left\| {Y \choose m} \right\|
<
\lp 1 + \sqrt{\frac{2}{ \ep \gamma}}  \rp
\frac{{l \choose m}}{{n \choose m}}
\left\| \FF \right\|.
\qed
\]
\end{theorem}

\medskip 

The rest of this subsection proves the theorem. 
Given such $\ep, m, l, \gamma$ and $\FF$, use the same $\DD$ and $\PP_j$ as Section 2.1. Find the following remarks. 

\be{A)} \setcounter{enumi}{5} 
\item Let the weight $w$ of $X$ be primitive with $\tilde{w}$, then 
\[
\left\| \GG \right\|  
= \sum_{V \in \GG 
} 
\tilde w \lp U \rp, 
\quad 
\textrm{for any~} \GG \subset 2^X, 
\] 
since 
$
\lbr \sum_{V \in \GG 
} 
\tilde w \lp U \rp \rbr^2 
= 
\sum_{\bs{U} \in \GG^2} w \lp \bs{U} \rp
= 
\| \GG \|^2 
$. 

\item By this {\em linearity of the norm $\| \cdot \|$ for primitive weight}, 
\[
\sum_{Y \in {X \choose l}} \left\| {Y \choose m} \right\| 
= 
\| \FF \| {n - m \choose l- m} 
= 
{n \choose m} {n - m \choose l- m} e^{-\kappa \lp \FF \rp}
= {n \choose l} {l \choose m} e^{-\kappa \lp \FF \rp}. 
\] 

\item 
$ 
\| \DD \| = \sum_{Y \in {X \choose l}} \left\| {Y \choose m} \right\|^2 
> 0
$, since 
\[
\|  \DD  \|  = 
\sum_{j=0}^m \| \PP_j \| {n- 2m+j \choose l - 2m+j}
< 
\sum_{j=0}^m h b^{-j} \| \FF \|^2  {n- 2m+j \choose l - 2m+j}
\] 
as in A). So $\| \FF \|>0$ meaning $\| \DD \|>0$ by definition. 

\item 
By \refeq{eqTildeG} for $g=2$,  
\[
\| \DD \| <
\frac{e^{-2 \kappa \lp \FF \rp}}{1- (2 \gamma)^{-1}}
 {n \choose m}{l \choose m}^2. 
\qed 
\] 
\ee 

\medskip 

We find another property on $\| \DD \|$ below. The statement is general holding for any $\FF$ that meets the conditions. 

\begin{lemma} \label{WeightBounds}
Let 
\be{i)}
\item $X$ be primitively weighted by $\lp 2^X \rp^2 \rightarrow \R_{\ge 0}$ 
inducing the norm $\| \cdot \|$, 
\item $l \in [n]$, $m \in [l]$, $t \in \R_{>0}$, 
\item $\FF \subset {X \choose m}$ such that 
\(
0<
\| \DD \|
\le t {n \choose l} {l \choose m}^2 e^{- 2 \kappa \lp \FF \rp},
\)
\item and $u, v \in \R_{>0}$ with 
\[
u<1, \hspace*{5mm}
u {n \choose l} \in \Z, \eqand
t < 1 + \frac{u (v-1)^2}{1-u}. 
\]
\ee 
The two statements hold.
\begin{enumerate} [a)]
\item If $v \ge 1$, more than $(1-u) {n \choose l}$ sets $Y \in {X \choose l}$ satisfy
$
\left\|  {Y \choose m} \right\| <   v  {l \choose m} e^{-\kappa \lp \FF \rp}.
$
\item If $v \le 1$, more than $(1-u) {n \choose l}$ sets $Y \in {X \choose l}$ satisfy
$
\left\|  {Y \choose m} \right\| >   v  {l \choose m} e^{-\kappa \lp \FF \rp}.
$
\end{enumerate}
\end{lemma}
\begin{proof}
a): Put
\[
z = e^{-\kappa \lp \FF \rp}, \eqand
x_j = \left\| {Y_j \choose m} \right\|,
\]
where $Y_j$ is the $j$th $l$-set in ${X \choose l}$. 
Suppose to the contrary that
$x_j \ge v z {l \choose m}$ if $1 \le j \le u {n \choose l}$.

Noting 
$\sum_{1 \le j \le {n \choose l}} x_j = z {n \choose l}{l \choose m}>0$ from G) and $\| \DD \|>0$, 
let $y \in (0, 1)$ satisfy
\[
\sum_{1 \le j \le u {n \choose l}} x_j  
= y z {n \choose l}{l \choose m},
\]
so
$
y \ge uv.
$ 
Find that 
\(
&&
\sum_{1 \le j \le u {n \choose l}} x_j^2
\ge \lbr \frac{y z {n \choose l}{l \choose m}}{u {n \choose l}} \rbr^2  u {n \choose l}
= \frac{y^2 z^2}{u}  {n \choose l}{l \choose m}^2,
\\ \textrm{and} &&
\sum_{u {n \choose l} < j \le {n \choose l}} x_j^2
\ge \lbr \frac{(1-y) z {n \choose l}{l \choose m}}{(1-u) {n \choose l}} \rbr^2  (1-u) {n \choose l}
= \frac{(1-y)^2 z^2 }{1-u} {n \choose l}{l \choose m}^2,
\)
meaning
\beeqn
&& \label{eqWeightBounds}
\| \DD \| = \sum_{Y \in {n \choose l}} \left\| {Y \choose m} \right\|^2 \ge f z^2 {n \choose l}{l \choose m}^2,
\eqwhere
f = \frac{y^2}{u} + \frac{(1-y)^2}{1-u}.
\eeqn
From $y \ge u v  \ge u$,
\beeq{eq2WeightBounds}
f \ge uv^2 + \frac{(1-uv )^2}{1-u} = 1 + \frac{u (v-1)^2}{1-u}>t.
\eeq
This contradicts the given condition proving a).

\medskip

b): Suppose $x_j \le  v z {l \choose m}$ if $1 \le j \le u {n \choose l}$. Use the same $y$ and $f$ so
$y \le uv$ and \refeq{eqWeightBounds}. These also imply \refeq{eq2WeightBounds} producing the same contradiction. Thus b).
\end{proof}

Set 
\(
t = \frac{1}{1- (2 \gamma)^{-1}},
\hspace*{5mm}
u = \frac{\lf \frac{\ep}{2} {n \choose l} \rf}{{n \choose l}},
\hspace*{5mm}
v = 1 + \frac{u}{\lp \frac{\ep}{2} \rp^{\frac{3}{2}}  \sqrt \gamma}, 
\eqand v' = 1- \frac{u}{\lp \frac{\ep}{2} \rp^{\frac{3}{2}}  \sqrt \gamma}. 
\)
Then 
\[
1 + \frac{u (v-1)^2}{1-u}= 1 + \frac{u (v'-1)^2}{1-u} 
= 1 + \gamma^{-1} \frac{u^3}{\lp \frac{\ep}{2} \rp^3 (1-u)}  
>t, 
\]
since $u > \frac{\ep}{2} - \ep^2$ from 
$\ep^{-2} \le \gamma \le l < {n \choose l}$. Here $l<n$ is assumed as the theorem is trivially true if $l=n$. 
By I) and \reflm{WeightBounds},
\[
v'   {l \choose m} e^{-\kappa \lp \FF \rp} <
\left\|  {Y \choose m} \right\|
<   v  {l \choose m} e^{-\kappa \lp \FF \rp},
\]
for some $\lp 1- 2 u \rp {n \choose l}$ sets $Y \in {X \choose l}$.
As $e^{-\kappa \lp \FF \rp} = \| \FF \|{n \choose m}^{-1}$, this means 
there are $\lc {n \choose l} \lp 1-  \ep \rp \rc$ sets $Y \in {X \choose l}$ such that
\[
\lp 1 - \sqrt{\frac{2}{\ep \gamma}}  \rp
\frac{{l \choose m}}{{n \choose m}}
\left\| \FF \right\|
<
\left\| {Y \choose m} \right\|
<
\lp 1 +\sqrt{\frac{2}{\ep \gamma}} \rp
\frac{{l \choose m}}{{n \choose m}}
\left\| \FF \right\|, 
\]
completing the proof of \refth{EGT42}.

\subsection{Deriving \refth{EGT}}
Given an $\FF$, let us assume for a while that $X$ is 
primitively weighted with $\tilde{w}: U \mapsto |\FF[U]|$. 
The norm of $\GG \subset 2^X$ and sparsity of $\FF$ in this default case are 
\[ 
\| \GG \| = \sum_{U \in \GG} |\FF[U]|, \eqand  
\kappa \lp \FF \rp = \ln {n \choose m} - \ln |\FF|, 
\] 
respectively, by the linearity of $\| \cdot \|$. 
The latter depends on $X$ as well as $|\FF|$. Generalize the default sparsity to any uniform family $\GG \in {X' \choose m'}$ of $m'$-sets in the universal set $X' \subset X$ $\lp m' \in \lbr |X' | \rbr \rp$ to write $\kappa \lp \GG \rp = \ln {|X'| \choose m'} - \ln |\GG|$. 

\medskip 

Remarks. 

\begin{enumerate} [A)] \setcounter{enumi}{9} 
\item The notation could be useful to express $\ln |\GG|$: for example, $|T| \le \kappa \lp \FF \rp\bigr/ \ln \frac{\ep l}{m^2 \lambda}$ for \refth{EGT} and $\kappa \lp \FF_{\boldsymbol X}\rp < \kappa \lp \FF \rp +m$ in \reflm{Split3} we will see in the next section. 
\item As we use on the bottom of the subsection, the same definition can apply to {\em the projection $\GG_{m'}$ of $\GG$ onto $X'$}, $i.e.$, $\GG_{m'} = \lb U \cap X'~:~ U \in \GG,~|U \cap X'|=m' |\rb$. 
\item $\kappa \lbr Ext \lp \FF, l \rp \rbr \le \kappa \lp \FF \rp$ for $l \in [n]-[m] $. For there are $|\FF| {n-m \choose l-m} = {n \choose m}e^{-\kappa \lp \FF \rp} {n-m \choose l-m} \\= {n \choose l}{l \choose m} e^{-\kappa \lp \FF \rp}$ set pairs $(S, T)$ such that $S \in \FF$, $T \in {X \choose l}$ and $S \subset T$. This means $|Ext \lp \FF, l \rp| \ge {n \choose l}e^{-\kappa \lp \FF \rp}$ leading to the claim.
\item Join a $p$-set $P$ to $X$ such that $P \cap X=\emptyset$.  The sparsity of $Ext \lp \FF, m+p \rp$ in the universal set $X \cup P$ is at most $\kappa \lp \FF \rp$ in $X$ since
\[
|Ext \lp \FF, m+p \rp|
\ge \sum_{j=0}^p {n \choose m+j}e^{-\kappa \lp \FF \rp} {p \choose p-j}  = {n+p \choose m+p} e^{-\kappa \lp \FF \rp} .
\]

\item The following lemma is proven in \cite{monotone} and Appendix 1.
\end{enumerate}

\begin{lemma} \label{PhaseII} 
For $\FF \subset{X \choose m}$ such that $m \le \frac{n}{2}$,
\[
\kappa \lbr {X \choose 2m} - Ext \lp \FF, 2m \rp \rbr
\ge
2 \kappa \lbr {X \choose m} - \FF \rbr. \qed
\]
\end{lemma}

\medskip

Assume $|\FF| > b^m$ for some $b \in \R_{\ge 1}$. There exists $T \subset X$ such that $|T|<m$,  $|\FF[T]| \ge |\FF| b^{-|T|}$, and $|\FF[T \cup S]| < b^{m-|T \cup S|}$ for any nonempty $S \subset X - T$. We use such a family $\FF[T]$ projected onto the universal set $X-T$ in place of $\FF$ in our proof of \refth{3SF}. 
Observe that the $\FF$ satisfies not only the $\Gamma(b)$-condition,  but also the $\Gamma_2 \lp b m^{-1},~ 1 \rp$-condition on $\| \cdot \|$ since 
\[ 
\| \PP_j \| 
= 
\sum_{U_1, U_2 \in \FF \atop |U \cap V| =j}  \tilde{w} (U_1) \tilde{w} (U_2) 
\le
\sum_{S \in {X \choose j}} |\FF[S]|^2
<  |\FF|^2 {m \choose j} b^{-j},
\]
for each $j \in [m]$.

By \refth{EGT42}:

\begin{corollary} \label{EGT4Cor}
Let $X$ be the universal set of cardinality $n$, $m \in [n-1]$, $l \in [n] - [m]$ and $\gamma \in \R_{>0}$ be sufficiently large not exceeding $\frac{l}{m}$. 
For any $\FF \subset {X \choose m}$ satisfying the $\Gamma \lp \frac{4  \gamma nm}{l} \rp$-condition, there are $\lc {n \choose l} \lp 1- \frac{2}{\sqrt[3] \gamma} \rp \rc$ sets $Y \in {X \choose l}$ such that
\[
\frac{{l \choose m} | \FF |}{{n \choose m}}
\lp 1 - \frac{1}{\sqrt[3] \gamma} \rp
<
\left| \FF \cap {Y \choose m} \right|
<
\frac{{l \choose m} | \FF |}{{n \choose m}}
\lp 1 + \frac{1}{\sqrt[3] \gamma} \rp. \qed
\]
\end{corollary}

\medskip

We show \refth{EGT} from the corollary. Given $m, l, \lambda$, sufficiently small $\ep$,  and $\FF$ as the statement, set
\[
l_0 = \lf \frac{l \sqrt \ep}{\lambda} \rf, \hspace*{5mm}
\gamma = \frac{1}{\sqrt[4] \ep}, \eqand
b = \frac{4 \gamma m n}{l_0}.
\]
Then $\gamma$ is sufficiently large and less than $\frac{l_0}{m^2}$ since
$
1 < \lambda < \frac{\ep l}{m^2}
$.

There exists a set $T$ such that $|T| \le \kappa \lp \FF \rp \Bigr/ \ln \frac{\ep l}{m^2 \lambda}$ and $\FF[T]$ satisfies the $\Gamma \lp b \rp$-condition in $X-T$: because the cardinality $j$ of such $T$ satisfies
\(
&& 
{n \choose m}  e^{-\kappa \lp \FF \rp} b^{-j}
= 
|\FF| b^{-j} \le |\FF[T]| \le {n-j \choose m-j},
\\ &\Rightarrow&
b^{-j}\lp \frac{n}{m} \rp^j
\le
b^{-j} \prod_{j'=0}^{j-1} \frac{n-j'}{m-j'}
=
\frac{{n \choose m}}{{n-j \choose m-j}} b^{-j}
\le e^{\kappa \lp \FF \rp},
\\ &\Rightarrow&
j \le \frac{\kappa \lp \FF \rp}{\ln \frac{\ep l}{m^2 \lambda}}.
\)
Assume $j<m$, otherwise the desired claim is trivially true.

Apply \refco{EGT4Cor} to $\FF[T]$ in the universal set $X-T$ noting
$\frac{l_0}{m^2} \le \frac{l_0-j}{(m-j)^2}$ and
$b \ge \frac{4 \gamma (m-j)(n-j)}{l_0-j}$.
We see 
\[
\left| Ext \lp \FF [T], l_0 \rp \right|
>
{n-j \choose l_0-j} \lp 1- \frac{2}{\sqrt[3] \gamma} \rp,
\]
from which
\[
\left| Ext \lp \FF[T], l \rp \right|
>
{n-j \choose l-j} \lp 1- e^{-\lambda} \rp,
\]
proving \refth{EGT}.  The truth of the last inequality is due to \reflm{PhaseII}: as $\frac{l - j}{l_0 -j} \ge \lambda \ep^{-1/2}$, it means
\[
\kappa \lbr {X \choose l}[T] - Ext \lp \FF[T], l \rp \rbr
\ge 2^{\lf \log_2 \frac{l-j}{l_0-j} \rf}~
\kappa \lbr {X \choose l_0}[T] - Ext \lp \FF[T], l_0 \rp \rbr
> \lambda,
\]
in the universal set $X - T$ leading to the inequality.

\section{Splitting the Universal Set}
Given $m \in [n]$ with $m \big| n$ and $q \in [m]$, let 
\[
d = \frac{nq}{m}, \eqand 
r = \lf \frac{m}{q} \rf. 
\]
Assume for a while 
\[ 
r \in \Z_{\ge 2}, 
\quad \Rightarrow \quad 
d  = \frac{n}{r}. 
\] 
Denote 
\[
\XX_j := \lb (X_1, X_2, \ldots, X_j)~:~  \textrm{$X_i$ are mutually disjoint $d$-sets} \rb, 
\]
for $j \in [r]$. 	
Call an element of $\XX_r$ {\em $r$-split of $X$} noting the given $q$ decides $r$. 

When a $j$ is also given, define 
\[
\FF_{\boldsymbol X} = \lb U~:~ U \in \FF, \textrm{~and~}
|U \cap X_{i}| = q \textrm{~for every~}  i \in [j] \rb. 
\] 
for $\FF \subset {X \choose m}$ and $\bs{X} \in \XX_j$, and 
\[
\TT_{\FF, j} := \lb
\lp U, {\boldsymbol X}\rp ~:~{\boldsymbol X} \in \XX_j \textrm{~and~}
U \in \FF_{\boldsymbol X}
\rb. 
\] 
Let $X$ be primitively weighted with $\tilde{w}: 2^X \rightarrow \R_{\ge 0}$, inducing the norm $\| \cdot \|$ and sparsity $\kappa$. Extend $\| \cdot \|$ to write 
\[
\| \TT_{\FF, j} \|  = \sum_{(U, \bs{X}) \in \TT_{\FF, j}} \tilde{w} (U). 
\]
For $j=0$, let $\XX_0 = \lb \emptyset \rb$ and $\FF_{\emptyset} = \FF$ so  $\|\TT_{\FF, 0} \|=\|\FF\|$. Assume $\| \FF \|>0$.

\medskip

Prove the following lemma.

\begin{lemma}\label{Split1} 
\[
\left\| \TT_{\FF, j}  \right\|
=
{d \choose q}^j {n - dj \choose m-qj} e^{-\kappa \lp \FF \rp}
\prod_{i=0}^{j-1}
{n - di  \choose d},
\]
for every $j \in [0, r) \cap \Z$ and $\FF \subset {X \choose m}$. 
\end{lemma}
\begin{proof}
We show the claim by induction on $j$ with the trivial basis $j=0$. Assume true for $j$ and prove for $j+1$. Fix any ${\boldsymbol X} = \lp X_1, X_2, \ldots, X_j \rp \in \XX_j$ putting
\[
X' = X - \bigcup_{i=1}^j X_i, \quad 
n' = |X'|, \quad 
m'= m-j q, \eqand
\gamma_{\boldsymbol X} = 
\frac{\left\| \FF_{\boldsymbol X} \right \|}
{{d \choose q}^j {n' \choose m'}}.
\]
Also write
\[
{\boldsymbol X}' = \lp X_1, X_2, \ldots, X_j, X_{j+1} \rp,
\]
for a $d$-set $X_{j+1} \in {X' \choose d}$.

For each $U \in \FF_{\boldsymbol X}$, there are
$
{m' \choose q} {n'-m' \choose d-q}
$
sets $X_{j+1} \in {X' \choose d}$ such that $\left| U \cap X_{j+1} \right|=q$. So the sum of $\tilde{w}(U)$ for $\lp U, {\boldsymbol X}' \rp \in \TT_{\FF, j+1}$ constrained by the fixed $\bs{X}$ 
is
\(
{m' \choose q}{n' - m' \choose d-q} \left\| \FF_{\boldsymbol X} \right\|
&=&
{m' \choose q} {n' - m' \choose d-q}  \gamma_{\boldsymbol X}  {d \choose q}^j {n' \choose m'} 
\nexteqline
{n' \choose q} {n'-q \choose m'-q} {n' - m' \choose d-q}  \gamma_{\boldsymbol X}  {d \choose q}^j. 
\)
Here 
\[
{n'-q \choose m'-q} {n' - m' \choose d-q} 
= \frac{(n'-q)!}{(m'-q)!(d-q)! (n-m'-d+q)!} 
= {n'-q \choose d-q} {n'-d \choose m'-q}. 
\]
So the above equals 
\[
{n' \choose q} {n'-q \choose d-q} {n'-d \choose m'-q} 
\gamma_{\boldsymbol X}  {d \choose q}^j 
=
\gamma_{\boldsymbol X}  
{d \choose q}^{j+1}  
{n'-d \choose m'-q}  {n' \choose d}. 
\]
Note $n' \ge m' + d -q$ from $\frac{m}{q} = r \ge j+1$ and $n'= n-dj$. 

By induction hypothesis,
\(
&&
\sum_{{\boldsymbol X} \in \XX_j} \left\| \FF_{\boldsymbol X} \right\|
=\left\|  \TT_{\FF, j} \right\|
=
{d \choose q}^j {n' \choose m'} e^{-\kappa \lp \FF \rp}
\prod_{i=0}^{j-1}
{n- di  \choose d},
\sothat
\sum_{{\boldsymbol X} \in \XX_j}  \gamma_{\boldsymbol X}
=
e^{-\kappa \lp \FF \rp}
\prod_{i=0}^{j-1}
{n-di  \choose d}.
\)
Hence,
\(
\left\| \TT_{\FF, j+1} \right\|
&=& \sum_{{\boldsymbol X} \in \XX_j}  
\gamma_{\boldsymbol X}  
{d \choose q}^{j+1}  
{n'-d \choose m'-q}  {n' \choose d} 
\nexteqline
{d \choose q}^{j+1} {n - d(j+1) \choose m-q(j+1)} e^{-\kappa \lp \FF \rp}
\prod_{i=0}^{j}
{n-di  \choose d},
\)
proving the induction step. The lemma follows.
\end{proof}

It means
$
\left\| \TT_{\FF, r-1}  \right\| =
{d \choose q}^r e^{-\kappa \lp \FF \rp}
\left| \XX_{r-1} \right|.
$
By the natural bijection between $\XX_{r-1}$ and $\XX_r$,
\[
\sum_{\bs{X} \in \XX_r} \left\| \FF_{\bs{X}} \right\|
=
\left\| \TT_{\FF, r}  \right\| =
{d \choose q}^r e^{-\kappa \lp \FF \rp}
\left| \XX_r \right|
= 
\frac{{d \choose q}^r \| \FF \|}{{n \choose m}} 
\left| \XX_r \right|
\]
Considering the case $r = 1$ as well, we have:

\begin{corollary} \label{Split2} 
Let $X$ be primitively weighted inducing the norm $\| \cdot \|$. 
Given $m \in [n]$ and $q \in [m]$, let 
$r = m / q$ and $d = n/r$ be both positive integers. 
For a family $\FF \subset {X \choose m}$ with $\| \FF \|>0$, there exists an $r$-split ${\boldsymbol X}$ of $X$ such that
$
\left\| \FF_{\boldsymbol X} \right\| \ge
{d \choose q}^r \|\FF\| \Bigr/ {n \choose m}
$.
\qed 
\end{corollary}

Note that if $q$ does not divide $m$ where $|U \cap X_r| = q'  \in [q, 2q) \cap \Z$ for $U \in \FF$, 
some $\bs{X}$ meets 
$\left\| \FF_{\boldsymbol X} \right\| \ge
{d \choose q}^{r-1} {n- d(r-1) \choose q'} \|\FF\| \Bigr/ {n \choose m}
$ by the same argument. For $q=1$: 

\begin{corollary} \label{Split3} 
For a universal set $X$ primitively weighted inducing the sparsity $\kappa$, and any family $\FF \subset {X \choose m}$ with $m \big| n$ and finite $\kappa \lp \FF \rp$, there exists an $m$-split $\bs{X}$ of $X$ such that 
$
\kappa \lp \FF_{\boldsymbol X} \rp 
<
\kappa \lp \FF \rp + m.
$ 
\end{corollary}
\begin{proof} 
Since $
\kappa \lp \FF_{\boldsymbol X} \rp \le \ln {n \choose m} - 
\ln \lbr \lp \frac{n}{m} \rp^m e^{-\kappa \lp \FF \rp} \rbr
<
\kappa \lp \FF \rp + m 
$ by the standard estimate of a binomial coefficient that is also derived in Appendix 2. 
\end{proof}

\medskip

Let us now focus on the first case $j=1$ of the lemma. Relax the constraints on $q$ and $d$ to see the following statement.

\begin{corollary} \label{Split4} 
Let $X$ be primitively weighted inducing $\| \cdot \|$, $m, d \in [n]$ and $q \in [0, m] \cap \Z$ such that $n-d-m+q \ge 0$.  
For each $\FF \subset {X \choose m}$ and $\ep \in (0, 1)$, there exist no more than $\lf \ep {n \choose d} \rf$ sets $X_1 \in {X \choose d}$ each with 
\[
\left\| \FF_{\bs{X}} \right\| > \frac{{n-d \choose m-q}{d \choose q}}{\ep {n \choose m}} \| \FF \|, 
\eqwhere 
\bs{X} = \lp X_1 \rp. 
\]
\end{corollary}
\begin{proof} 
Such objects can apply to the case $j=1$ of \reflm{Split1}. 
The weight sum for all tuples $(U, \bs{X}) \in \TT_{\FF, 1}$ is exactly 
\[
\| \TT_{\FF, 1} \| 
= 
{d \choose q} {n - d\choose m-q} e^{-\kappa \lp \FF \rp}
{n  \choose d} 
= 
\frac{{n-d \choose m-q}{d \choose q}}{{n \choose m}}  \| \FF \|
{n \choose d}
\] 
The claim follows. 
\end{proof}

\section{Proof of \refth{3SF}} 
We prove the theorem in this section. Given $\FF \subset {X \choose m}$ and a sufficiently small $\delta \in (0, 1/2)$ by the statement, let 
\[
\ep = e^{-1/\delta}, \quad 
g = \lf e^{1/\ep} \rf, \quad 
c = e^g, \eqand 
b_*= e^c m^{\frac{1}{2} + \delta}, 
\]
assuming $\FF$ satisfies the $\Gamma \lp b_* \rp$-condition. 
WLOG $m > e^{\ep c}$, otherwise $\FF$ includes three mutually disjoint sets similarly to the proof of the sunflower lemma \cite{OriginalSF}: select any $U_1 \in \FF$ eliminating all sets in $\FF$ that intersect with $U_1$. By the $\Gamma\lp b_* \rp$-condition with $b_* > 3m$, this removes less than a third of the original $\FF$. Find $U_2$ and $U_3$ in the remaining $\FF$ similarly, and the obtained three are mutually disjoint.

Further let 
\[ 
z= \lc \log_2 m^{(1-\delta)/2} \rc, \quad 
r = 2^z, \eqand 
q  = \frac{m}{r}. 
\]
Assume $n=|X|$ is larger than $m^4$ and divisible by $m r$. Otherwise add some extra elements to $X$. 

\subsection{Preprocess} 
On such objects, we first perform our initial construction. Prove a recursive statement.

\begin{lemma} 
Let 
\be{i)}
\item $j \in [0, z] \cap \Z$, $\delta_j = \sum_{j'=0}^j \lp 2^{-j'} m \rp^{\frac{1}{2} + \ep} $, 
\item $X' \in {X \choose 2^{-j} n}$, weighted primitively inducing the norm $\| \cdot \|$, 
\item and $\GG \subset 2^{X'}$ such that $\| \GG \|>0$, and $\big|~ |U| - 2^{-j} m ~\big| < \delta_j $ for every $U \in \GG$. 
\ee 
There exists an $2^{z-j}$-split $\bs{X} =\lp X_1, X_2, \ldots, X_{2^{z-j}} \rp$ of $X'$, and $\GG' \subset \GG$ such that 
\[
\| \GG'  \| > \lp 1 - 4^{z -j} e^{-m^\ep} \rp \| \GG \|, \eqand 
\Big|~ |U \cap X_{j'}| - 2^{-z} m ~\Big| < \delta_z, 
\]
for every $j' \in \lbr 2^{z-j} \rbr$ and $U \in \GG'$. 
\end{lemma}
\begin{proof} 
Proof by induction on $j$ with the trivial basis $j= z$. Assume true for $j+1$ and prove true for $j$.

Let 
\[
\GG_{m_1, m_2, Y} = 
\lb U~:~ U \in \GG,~ |U \cap Y|=m_1, \textrm{~and~} 
|U \cap X'-Y|=m_2 
\rb, 
\] 
for each $Y \in {X' \choose |X'|/2}$, and $m_1, m_2 \in [m]$ with $|m_1+m_2 - 2^{-j} m| < \delta_j$. By \refco{Split4}, there are no more than $m^{-3} {|X'| \choose |X'|/2}$ sets $Y$ such that 
\beeq{eq2Initial} 
\| \GG_{m_1, m_2, Y} \|  > \frac{m^3 {|X'|/2 \choose m_1}{|X'|/2 \choose m_2}}{{|X'| \choose m_1 + m_2}} \| \GG \|. 
\eeq

We also have 
\(
&&
\ln \frac{{|X'|/2 \choose m_1}{|X'|/2 \choose m_2}} {{|X'| \choose m_1 + m_2}} 
< - \frac{\lp m_1 - \frac{m_1 + m_2}{2} \rp^2}{2 (m_1+m_2)} 
< -m^{\ep + \ep^2}, 
\\  &&
\textrm{if} \quad 
|m_1-m_2|  > m^{-\ep^2} (m_1+m_2)^{\frac{1}{2} +\ep},  
\)
by \reflm{SpaceProp}, since 
$
m_1 + m_2 > 2^{-z+1} m - \delta _j > m^{\frac{1+\delta}{2}} 
$. 

For every possible combination of $m_i$, exclude all $Y$ with \refeq{eq2Initial} from consideration.  Fix any one remaining $Y$, and all $\GG_{m_1, m_2, Y}$ meet $\neg$ \refeq{eq2Initial}. Delete from $\GG$ the union of $\GG_{m_1, m_2, Y}$ each with $|m_1-m_2|  > m^{-\ep^2} (m_1+m_2)^{\frac{1}{2} +\ep}$. Then 
\bdash 
\item this reduces $\| \GG \|$ only by its $e^{-m^\ep}$ or less, 
\item and $
|m_i - 2^{-j-1} m|  < \delta_{j+1}
$ for each remaining $\GG_{m_1, m_2, Y}$ and $i \in [2]$, since $| m_1+ m_2 - 2^{-j} m | < \delta_j$ and $|m_1-m_2| \le m^{-\ep^2} (m_1+m_2)^{\frac{1}{2} +\ep}$. 
\edash 

\medskip 

Now we obtain recursive solutions in both $Y$ and $X' - Y$. 
Weight $Y$ primitively with $U \mapsto \| \GG[U] \|$.  
Apply 
$
\GG_1  = \lb U \cap Y~:~ U \in \GG \rb 
$ 
to the induction hypothesis to obtain an $2^{z-j-1}$-split $\bs{X}_1 =  \lp X_1, X_2, \ldots, X_{2^{z-j-1}} \rp$ of $Y$ and subfamily $\GG' \subset \GG$ such that 
$
\| \GG' \| > \lp 1 - 4^{z-j-1} e^{-m^\ep} \rp \| \GG \|, 
$ 
and  $|U \cap X_{j'} - 2^{-z} m| < \delta_z$ for every $j' \in \lbr 2^{z-j-1} \rbr$ and $U \in \GG'$. 

Replace $\GG$ by $\GG'$. Similarly construct $\GG_2 \subset 2^{X'-Y}$ to obtain a $2^{z-j-1}$-split $\bs{X}_2$ of $X'-Y$ and new $\GG'$ that satisfy the two conditions. 

Concatenate the two splits $\bs{X}_i$ to construct the $2^{z-j}$-split $\bs{X}$ of $X'$. As 
\[
\lp 1 - 4^{r-j-1} e^{-m^\ep} \rp^2 \lp 1- e^{-m^\ep} \rp 
> \lp 1 - 4^{r-j} e^{-m^\ep} \rp, 
\]
the obtained $\bs{X}$ and $\GG'$ meet the two desired conditions.  
We have proven the induction step completing the proof. 
\end{proof}

For $\GG= \FF$ in $X$ weighted primitively with $U \mapsto | \FF[U]|$, obtain such an $r$-split $\bs{X} = \lp X_1, X_2, \ldots, X_r \rp$ of $X$ and $\GG'$ by the lemma. 
Replace $\FF$ by $\GG'$ and $b_*$ by $b_*/2$, then the new $\FF$ satisfies $\big|~|U \cap X_j| - q ~\big|< \delta_z$ for each $j \in [r]$ and $U \in \FF$, in addition to all the conditions seen above.

\medskip

We now construct three sets $C_i$ and subfamilies $\FF_i \subset \FF$ $\lp i \in [3] \rp$ by a recursive process with the index $j \in [r]$: initially set $C_i=\emptyset$ and $\FF_i=\FF$ for all $i$. 
At the beginning of the $j$th trial, we are given $C_i$ and $\FF_i$ with 
$
|C_i| < j q m^{-\ep}, 
$ and the $\Gamma \lp b_j, 2\rp$-condition of $\FF_i$, $i.e.$, 
$|\FF_i[S]| < 2 b_j^{-|S|} |\FF_i|$ for every nonempty $S \subset X$, where 
\[
b_j = \ep b_*  \lp 1 -  \frac{1}{r} \rp^{j-1}. 
\] 
Putting 
\[
Q = \lbr  (1-\ep)q,~ (1+\ep)q \rbr \cap \Z, 
\] 
find and fix $q_{i, j} \in Q$ 
such that $\left| U~:~ U \in \FF_i,~ |U \cap X_j| = q_{i, j}\right|$ is maximum. 
Also let $S_i$ be a maximal set in $X- C_i$ such that 
$
|\FF_i[S_i]| \ge b_j^{-|S|} |\FF_i| 
$. Update $\FF_i$ and $C_i$ by 
\[
\FF_i \leftarrow \FF_i[S_i], \eqand 
C_i \leftarrow C_i \cup S_i, 
\]
where $\leftarrow$ represents substitution for update. 
Let the other two $\FF_{i'}$ $\lp i' \in [3] - \lb i \rb \rp$ {\em exclude $S_i$}, $i.e.$, update them by 
$
\FF_{i'} \leftarrow \FF_{i'} \cap {X - S_i \choose m} 
$. 
Also performing $j \leftarrow j+1$, continue to the next trial if $j \le r$. This completes the description of our recursive process. 

\medskip 

Right after the update $\FF_i \leftarrow \FF_i[S_i]$, we have $|S_i| < q m^{-\ep}$ and the $\Gamma \lp b_j \rp$-condition of $\FF_i$; by the $\Gamma \lp b_{j-1}, 2 \rp$-condition given at the begining of the $j$th trial, 
\[
|\FF_i[S]| < \frac{2 (1 + \ep) q}{b_{j-1}^{|S|}} |\FF_i| 
< m \lp 1- \frac{1}{r} \rp^{|S|} ~b_j^{-|S|} |\FF_i|,  
\]
so $|S_i|$ must be less than $q m^{-\ep}$ while the new $\FF_i$ satisfies the $\Gamma(b_j)$-condition. After excluding $S_{i'}$ of the other two $\FF_{i'}$, the family $\FF_i$ correctly satisfies the $\Gamma(b_j, 2)$-condition. 

By these, the obtained objects safisfy that: 
\be{A)}
\item $C_i$ are three mutually disjoint sets each with $|C_i|< m^{1-\ep}$, 
\item 
$
\FF_i \subset \FF[C_i] \cap {X - \bigcup_{i' \in [3]- \lb i \rb} C_{i'} \choose m} 
$ 
with 
$|\FF_i|> m^{-r}~b_1^{-|C_i|} |\FF|$ and the $\Gamma \lp b_{r-1},~ 
2 \rp$-condition, 
\item 
and $|U \cap X_j| = q_{i, j} \in Q$ for each $i \in [3]$, $j \in [r]$ and $U \in \FF_{i}$. 
\qed 
\ee

\subsection{Recursive Updates on $\bs{X}$}  
Put $\FF_{i, 0} = \FF_i$ and $C_{i, 0} = C_i$ freeing the variables $\FF_i$ and $C_i$. Also update $b_* \leftarrow \ep b_{r-1}$ with which we use the same $b_j$ as above. The families satisfy 
$|\FF_{i, 0}| > b_*^{m-|C_{i, 0}|}$ and the $\Gamma \lp b_* \rp$-condition\footnote{
We say $\GG \subset {X \choose m}$ satisfies the $\Gamma \lp b_* \rp$-condition in $X' \subset X$ if $|\GG[S]| < b_*^{-|S|} |\GG|$ for every nonempty $S \subset X'$. 
} in $X-C_{i, 0}$, embedded in $\bs{X}$ the way C) describes.

We show the following property for every $j \in [r+1]$.

\medskip

{\noindent\bf Property $\Pi_j$}: there exist three mutually disjoint sets $C_i \supset C_{i. 0}$ and subfamilies $\FF_i \subset \FF_{i. 0}$ satisfying the following conditions. 
\begin{enumerate}[i)]
\item 
$
\FF_i \subset \FF_{i, 0}[C_i] \cap {X - \bigcup_{i' \in [3]- \lb i \rb} C_{i'} \choose m} 
$
such that $|\FF_i| >  \ep^{j q} (\ep b_*)^{-|C_i| + |C_{i, 0}|} |\FF_{i, 0}|$.  
\item If $j \le r$, 
	\balph
	\item 
	\[
	\sum_{ 
	u \in [r, m] \cap \Z \atop 
	S \in {Z_j \choose u} 
	}
	\frac{|\FF_i[S]|^g}{b_j^{-(g-1)u} {m_* \choose u}}
	< |\FF_i|^g, 
	\]
	where 
	\[
	m_* = (g-1) m, \eqand 
	Z_j =\bigcup_{p=j}^r X_p - C_i, 
	\] 
	\item and the $\Gamma (b_j m^{-\ep})$-condition of $\FF_i$ in $Z_j$. 
	\ealph
\item $U \cap U' \cap \bigcup_{j' \in [j-1]} X_{j'} =\emptyset$ for each $U \in \FF_i$ and $U' \in \FF_{i'}$ with $i' \in [3]- \lb i \rb$.  
\end{enumerate}

\medskip 

As $\Pi_{r+1}$-iii) means three mutually disjoint sets in $\FF$, our task here is to prove $\Pi_j$ by induction on $j$. For the basis $j=1$,  choose $C_i =C_{i, 0}$ and $\FF_i=\FF_{i, 0}$ satisfying $\Pi_1$-i) to iii). Here $\Pi_1$-ii)-a) holds since $\sum_{S \in {X \choose u}} |\FF_i[S]|^g < b_*^{-(g-1)u} {m \choose u} |\FF_i|^g$ for every $u \in [m]$, by the $\Gamma(b_*)$-condition of $\FF_i$. This confirms the basis.

Assume $\Pi_j$ and prove $\Pi_{j+1}$. 
When we are given $C_i$ and $\FF_i$ of $\Pi_j$, write for simplicity 
\(
&&
b = b_j, \quad  
Z= Z_{j+1}, \quad 
X_* = X_j - C_i,  
\\ && 
q_* = q_{i, j}- |X_j \cap C_i|, \quad 
n_* = |X_*|, 
\eqand 
\HH = {X_* \choose q_*}. 
\) 
We may use $s, t, u \in \Z_{\ge 0}$ as summation/product indices. 
Obvious floor functions are omitted in the rest of the proof.

The induction step will update $C_i$ and $\FF_i$ given by $\Pi_j$, so they satisfy $\Pi_{j+1}$. We complete it in seven steps.

\medskip

\noindent
{\bf Step 1.~}{\em Construct a family $\YY_i$ of $Y \in {X_* \choose n_*/ 4}$ such that $\FF_i \cap {X - X_* \cup Y \choose m}$ is sufficiently large. } Fix each $i \in \lbr 3 \rbr$ assuming $q_*>0$. Weight $X_*$ by $w : \lp 2^{X_*} \rp^2 \rightarrow \Z_{\ge 0}$ primitively with 
$
V \mapsto |\FF_i [V]|
$,  
inducing the norm $\| \cdot \|$. 
Then the family  $\HH$ satisfies the $\Gamma_2 \lp \frac{b}{q_* m^{\ep}}, 1 \rp$-condition on $\| \cdot \|$, since 
\(
\sum_{V_1, V_2 \in \HH \atop |V \cap V'| = u} w(V_1, V_2) 
&=& 
\sum_{V_1, V_2 \in \HH \atop |V \cap V'| = u} |\FF_i[V_1]|~|\FF_i[V_2]| 
\le
\sum_{T \in {X_* \choose u}} |\FF_i [T]|^2
\\ &<&
|\FF_i|^2  \lp b m^{-\ep} \rp^{-u} {q_* \choose u}
\le 
\lp \frac{b}{q_* m^{\ep}} \rp^{-u} \| \HH \|^2  
, 
\) 
for each $u \in [q_*]$, by $\Pi_j$-ii)-b) and $\| \HH \| = |\FF_i|$.

Apply \refth{EGT42} to $\HH$. 
There exists a family $\YY_i \subset {X_* \choose n_*/4}$ such that
$
|\YY_i| > {n_* \choose n_*/4} \lp 1  - \ep \rp
$,
and 
\beeqn 
&& \label{eqStep1}
\left| \FF_{Y, i} \right|
>
\frac{{n_*/4 \choose q_*}}{{n_* \choose q_*}} |\FF_i|
\lp 1  - \ep \rp, \quad 
\textrm{for every~} Y \in \YY_i, 
\\ && \nonumber 
\textrm{where} \quad \FF_{Y, i} := \FF_i \cap {X - X_* \cup Y \choose  m}.
\eeqn

\medskip

\noindent
{\bf Step 2.~}{\em With another weight $w$ on $X_*$, confirm some $\Gamma_g$-condition of $\HH$.} 
For each $i$, skip this step, Steps 3, 4 and 6 if $j=r+1$ or $q_*=0$. 
Denote by $S$ a subset of $X_*$, and by $T$ a nonempty subset of $Z$. 
Define 
\( 
w_T: \lp 2^{X_*} \rp^g \rightarrow \R_{\ge 0}, \quad 
(V_1, V_2, \ldots, V_g) \mapsto 
\frac{\prod_{t=1}^g \left| \FF_i[V_t \cup T] \right|}
{
b^{-(g-1)|T|} {m_* \choose |T|} 
}, 
\) 
for each $T$ inducing the norm $\| \cdot \|_T$. Reset $w$ and $\| \cdot \|$ by 
\(
w: \lp 2^{X_*} \rp^g \rightarrow \R_{\ge 0}, \quad 
\bs{V} \mapsto
\sum_{r \le u \le m \atop
T \in {Z \choose u} }  
w_T \lp \bs{V} \rp. 
\) 
Also denote 
\(
&& 
w_{S, T}: = 
\frac{\left| \FF_i[S \cup T] \right|}
{
b^{-\lp 1- \frac{1}{g}\rp |T|} {m_* \choose |T|}^{\frac{1}{g}}
},  
\quad 
\textrm{for each $S$ and $T$},   
\\ && 
\gamma_T := |\FF_i|^{-g} 
\sum_{0 \le s \le q_* \atop S \in {X_* \choose s}} 
\frac{\left| \FF_i[S \cup T] \right|^g}
{
b^{-(g-1) \lp s + |T| \rp} {m_*  \choose {s + |T|}}
}, 
\quad 
\textrm{for each $T$},   
\\ && 
b_g := \frac{b^{1- \frac{1}{g}}}{2^g m_*^{\frac{1}{g}}}, \quad 
b_\dagger := \frac{b_g}{(g-1) q_*}, 
\eqand 
h := \lp \frac{|\FF_i|}{\left\| \HH \right\|} \rp^g. 
\) 
This step shows the $\Gamma_g \lp b_\dagger, h \rp$-condition of $\HH$ on $\| \cdot \|$. 

\medskip

See the following remarks. 
\be{A)} \setcounter{enumi}{3}
\item  For each $S$ and $T$, 
\[
\sum_{(V_1, V_2, \ldots, V_g) \in \HH[S]^g}
~
\prod_{t=1}^g \left| \FF_i[V_t \cup T] \right| 
= \left| \FF_i[S \cup T] \right|^g, 
\]
so 
\[
\left\| \HH[S] \right\|^g =
\sum_{\bs{V} \in \HH[S]^g} w \lp \bs{V} \rp
=
\sum_{r \le u \le m \atop
T \in {Z \choose u} } 
\sum_{\bs{V} \in \HH[S]^g} w_T \lp \bs{V} \rp
=
\sum_{r \le u \le m \atop
T \in {Z \choose u} } 
\frac{|\FF_i[S \cup T]|^g  
}
{
b^{-(g-1)u} {m_* \choose u}
}. 
\]

\item $\sum_{r \le u \le m \atop T \in {Z \choose u}} \gamma_T < 1$ by $\Pi_j$-ii)-a).

\item For each $T$, 
\(
&& 
\sum_{0 \le s \le q_* \atop S \in {X_* \choose s}} 
\frac{\| \HH[S] \|_T^g}{b^{-(g-1)s}~{m_* \choose s} } 
=
\sum_{0 \le s \le q_* \atop S \in {X_* \choose s}} 
\sum_{\bs{V} \in \HH[S]^g}  
\frac{w_T \lp \bs{V} \rp}{b^{-(g-1)s}~{m_* \choose s} } 
\\ &\le& 
\sum_{0 \le s \le q_* \atop S \in {X_* \choose s}} 
\frac{\left| \FF_i[S \cup T] \right|^g}
{
b^{-(g-1)\lp s + |T| \rp} {m_*  \choose s + |T|}
}
= 
\gamma_T |\FF_i|^g. 
\)
The inequality holds by D) and 
\[ 
{m_* \choose |T|} {m_* \choose s} 
\ge 
{m_* -s \choose |T|} {m_* \choose s} 
= 
{m_* \choose s+|T|}{s+|T| \choose s}
\ge 
{m_* \choose s+|T|}. 
\]

\item 
$h$ is greater than 1, otherwise 
\(
|\FF_i|^g &\le& \| \HH \|^g 
=
\sum_{r \le u \le m \atop
T \in {Z \choose u}} 
\frac{|\FF_i[T]|^g  
}
{
b^{-(g-1)u} {m_* \choose u}
} 
< 
|\FF_i|^g, 
\)
by D) and $\Pi_j$-ii)-a). 

\item For each $S$ and $T$, 
\[ 
w_{S, T} = \| \HH[S] \|_T \le 
\gamma_T^{\frac{1}{g}} |\FF_i| b^{ - \lp 1- \frac{1}{g} \rp |S|} m_*^\frac{|S|}{g},  
\]
due to F) and 
\[
w_{S, T}^g = 
\frac{|\FF_i[S \cup T]|^g}{b^{-(g-1)|T|} {m_* \choose |T|}} 
=
\sum_{\bs{V} \in \HH[S]^g} w_T \lp \bs{V} \rp 
= \| \HH[S] \|_T^g. 
\qed 
\]
\ee

\medskip 

Let us show the $\Gamma_g$-condition with the remarks. It suffices to confirm 
\beeq{eqMainStep6}
\| \PP_{s, g} \|_T < \gamma_T b_g^{-s} {(g-1)q_* \choose s} |\FF_i|^g, 
\eeq
for every $T$ and $s \in [(g-1) q_*]$: it is due to E) and 
\[
\| \PP_{s, g} \| 
= 
\sum_{\bs{V} \in \PP_{s, g}} w \lp \bs{V} \rp 
= 
\sum_{\bs{V} \in \PP_{s, g}} 
\sum_{
r \le u \le m_* \atop T \in{Z \choose u}
} 
w_T \lp \bs{V} \rp 
= 
\sum_{
r \le u \le m_* \atop T \in{Z \choose u} 
} 
\| \PP_{s, g} \|_T. 
\]
Here $\PP_{s, g}$ is defined for $\HH$ as in Section 2, $i.e.$, 
\beeq{eqPPStep6}
\PP_{s, g} = \lb \bs{V} ~:~ \bs{V} \in \HH^{g},~ 
|union(\bs{V})| = g q_* - s
\rb, 
\eeq 
for every $s \ge 0$. 
So if we show \refeq{eqMainStep6} for all $T$ and $s$, we have the $\Gamma_g \lp b_\dagger, h \rp$-condition of $\HH$ on $\| \cdot \|$.

Fix each $T$ for the proof. Define 
\[
w_{g'}: \lp 2^{X_*}\rp^{g'} \rightarrow \R_{\ge 0}, \quad 
(V_1, V_2, \ldots, V_{g'}) \mapsto 
\frac{\prod_{t=1}^{g'} \left| \FF_i[V_t \cup T] \right|}
{
b^{-\lp 1-\frac{1}{g}\rp g' |T|} {m_* \choose |T|}^{\frac{g'}{g}}
},  
\] 
for $g' \in [2, g] \cap \Z$ inducing the norm $\| \cdot \|_{g'}$. We verify 
\beeq{eqPfStep6}
\| \PP_{s, g'} \|_{g'} < \gamma_T^{\frac{g'}{g}}  b_{g'}^{-s} {(g'-1)q_* \choose s} 
|\FF_i|^{g'}, 
\eeq 
for every $g'$ and $s$, where 
$b_{g'}:= 2^{g-g'} b_g$, and 
$\PP_{s, g'}$ is given by replacing $g$ by $g'$ in \refeq{eqPPStep6}. 
The case $g' = g$ means \refeq{eqMainStep6}. 

\medskip 

Proof of \refeq{eqPfStep6} by induction on $g'$. Fix each $s \in [(g'-1) q_*]$ for the basis $g'=2$. 
From H) above, 
\(
\| \PP_{s, 2} \|_2 &=& \sum_{\bs{V} \in \PP_{s, 2}} w_2 \lp  \bs{V} \rp  
= 
\sum_{(V_1, V_2) \in \PP_{s, 2}} 
\frac{|\FF_i[V_1 \cup T]|~|\FF_i[V_2 \cup T]|}{
b^{-\lp 1- \frac{1}{g} \rp 2|T|} {m_* \choose |T|}^{\frac{2}{g}} 
} 
\\ &\le& 
\sum_{V_1 \in \HH} \frac{|\FF_i[V_1 \cup T]|}{
b^{-\lp 1- \frac{1}{g} \rp |T|} {m_* \choose |T|}^{\frac{1}{g}} 
} 
\sum_{V_2 \in \HH \atop \textrm{with~} |V_1 \cup V_2|=2q_* - s} \frac{|\FF_i[V_2 \cup T]|}{
b^{-\lp 1- \frac{1}{g} \rp |T|} {m_* \choose |T|}^{\frac{1}{g}} 
} 
\\ &\le& 
w_{\emptyset, T}~ 
{q_* \choose s} \max_{S \in {X_* \choose s}} w_{S, T}  
\\ &\le&
\gamma_T^{\frac{2}{g}}  |\FF_i|^2 b^{-\lp 1- \frac{1}{g}\rp s} m_*^{\frac{s}{g}} {q_* \choose s}
\\ &<& 
\gamma_T^{\frac{2}{g}}   b_2^{-s} 
{q_* \choose s} |\FF_i|^2
, 
\)
proving the basis.

Assume true for $g'-1$ and prove true for $g'$. By induction hypothesis, 
\[
\sum_{\bs{V} \in \PP_{v, g'-1} 
} 
w_{g'-1} \lp \bs{V} \rp
=
\| \PP_{v, g'-1} \|_{g'-1} 
< 
\gamma_T^{\frac{g'-1}{g}} b_{g'-1}^{-v} {(g'-2)q_* \choose v} |\FF_i|^{g'-1}, 
\]
for $v \in [(g'-2)q_*]$. Fix any $s \in [(g'-1)q_*]$. Since $\PP_{v, g'-1}=\emptyset$ if $v > (g'-2) q_*$, 
\(
&& 
\sum_{v \in [s] \atop \bs{V} \in \PP_{v, g'-1}}  w_{g'-1} (\bs{V})  
\sum_{S \in {union(\bs{V}) \choose s-v}} w_{S, T} 
\\ &\le& 
\sum_{v \in \min \lbr s,~(g'-2) q_*\rbr \atop \bs{V} \in \PP_{v, g'-1} 
} 
w_{g'-1} \lp \bs{V} \rp 
~
{(g'-1) q_* -v  \choose s- v} 
\max_{S \in {union(\bs{V}) \choose s-v}} w_{S, T} 
\\ &<& 
\sum_{v=1}^{s}
\gamma_T^{\frac{g'-1}{g}} b_{g'-1}^{-v} {(g'-2)q_* \choose v} |\FF_i|^{g'-1} 
\\ && \quad 
\cdot~ 
{(g'-1) q_* -v  \choose s- v} 
\gamma_T^{\frac{1}{g}} |\FF_i| b^{ - \lp 1- \frac{1}{g} \rp (s-v)} m_*^\frac{s-v}{g}
\\ &<& 
\gamma_T^{\frac{g'}{g}} \lp 2 b_{g'} \rp^{-s} 
{(g'-1) q_* \choose s} 
|\FF_i|^{g'} 
\sum_{v=1}^{s}
{s \choose v}.  
\)
The last line is due to 
$
{(g'-2)q_* \choose v} {(g'-1) q_* -v  \choose s- v} 
< 
{(g'-1)q_* \choose v} {(g'-1) q_* -v  \choose s- v} 
\\ 
= 
{(g'-1)q_* \choose s}{s \choose v}
$ for every $v$.

For $v=0$, we have 
\[
\sum_{\bs{V} \in \PP_{0, g'-1}}  w_{g'-1} (\bs{V})  
\le w_{\emptyset, T}^{g'-1}
\le \gamma_T^{\frac{g'-1}{g}} |\FF_i|^{g'-1},  
\]
by H), so 
\(
\sum_{\bs{V} \in \PP_{0, g'-1}}  w_{g'-1} (\bs{V})  
\sum_{S \in {union(\bs{V}) \choose s}} w_{S, T} 
< \gamma_T^{\frac{g'}{g}} \lp 2 b_{g'} \rp^{-s} 
{(g'-1) q_* \choose s} 
|\FF_i|^{g'}. 
\) 

As $w_{S, T}= 
\frac{\sum_{V \in \HH[S]} \left| \FF_i[V \cup T] \right|}
{
b^{-\lp 1- \frac{1}{g}\rp |T|} {m_* \choose |T|}^{\frac{1}{g}}
}$ for each $S$, we conclude that 
\(
\| \PP_{s, g'} \|_{g'}  
&=&  \sum_{\bs{V} \in \PP_{s, g'}}  w_{g'} (\bs{V}) 
= 
\sum_{(V_1, V_2, \ldots, V_{g'}) \in \PP_{s, g'}} 
~
\frac{\prod_{t=1}^{g'} \left| \FF_i[V_t \cup T] \right|}
{
b^{-\lp 1-\frac{1}{g}\rp g' |T|} {m_* \choose |T|}^{\frac{g'}{g}}
} 
\\ &\le& 
\sum_{0 \le v \le s \atop \bs{V} \in \PP_{v, g'-1}}  w_{g'-1} (\bs{V})  
\sum_{S \in {union(\bs{V}) \choose s-v}} w_{S, T} 
\\ &<& 
\gamma_T^{\frac{g'}{g}} \lp 2 b_{g'} \rp^{-s} 
{(g'-1) q_* \choose s} 
|\FF_i|^{g'} 
\sum_{v=0}^{s}
{s \choose v} 
\nexteqline 
\gamma_T^{\frac{g'}{g}} b_{g'}^{-s} 
{(g'-1) q_* \choose s} 
|\FF_i|^{g'}, 
\)
completing the induction step.

This confirms \refeq{eqPfStep6}, hence the $\Gamma_g \lp b_\dagger, h \rp$-condition of $\HH$ on $\| \cdot \|$ as well.

\medskip 

\noindent
{\bf Step 3.~}{\em Remove  $Y$ from $\YY_i$ such that 
$\sum_{T \in {Z \choose u}} |\FF_{Y, i}[T]|^g$ 
is too large for any $u \ge r$.} 
With the $\Gamma_g$-condition meeting 
\[
b_\dagger > \frac{4 g^2 n_*}{n_*/4}, \eqand 
h>1 \textrm{~ from G)}, 
\] 
apply \refco{EGTTildeG} to $\HH$. There are $\lc \lp 1-  \ep \rp {n_* \choose n_*/4} \rc$ sets $Y \in {X_* \choose n_*/4}$ such that
\(
&&
\sum_{r \le u \le m \atop
T \in {Z \choose u} } 
\frac{|\FF_{Y, i}[T]|^g  
}
{
b^{-(g-1)u} {m_* \choose u}
}
= 
\sum_{\bs{V} \in \lbr \HH \cap {Y \choose q_*} \rbr^g} w \lp \bs{V} \rp 
\\ &<&
\frac{\lp 1 + \frac{h}{g} \rp {n_*/4 \choose q_*}^g}{\ep  {n_* \choose q_*}^g}
\| \HH \|^g 
< 
\frac{2 {n_*/4 \choose q_*}^g}{\ep  {n_* \choose q_*}^g}
|\FF_i|^g
\\ &<&
\frac{3}{\ep} |\FF_{Y, i}|^g,
\)
since \refeq{eqStep1}. As 
$b = b_{j+1} \lp 1 - r^{-1} \rp^{-1}$, 
the inequality means 
\beeq{eqStep3}
\sum_{ 
r \le u \le m \atop
T \in {Z \choose u} 
}
\frac{|\FF_{Y, i}[T]|^g}{b_{j+1}^{-(g-1)u} {m_* \choose u}}
< \ep |\FF_{Y, i}|^g. 
\eeq
Delete $Y$ such that $\neg$\refeq{eqStep3} from $\YY_i$. Now the family satisfies $|\YY_i| > \lp 1- 2 \ep \rp {n_* \choose n_*/4}$, and \refeq{eqStep1} $\wedge$ \refeq{eqStep3} for every $Y \in \YY_i$.

\medskip

\noindent
{\bf Step 4.~}{\em Further delete some undesired $Y$ from $\YY_i$.}  
Reset $b$ by $b \leftarrow \ep b_j m^{-\ep}$. Let $w_T$, $\| \cdot \|_T$, $w_{S, T}$, $\gamma_T$, and $b_g$ be the same as Step 2 with the updated $b$. Also let 
\(
w: \HH^g \rightarrow \R_{\ge 0}, \quad 
\bs{V} \mapsto
\sum_{
T \in {Z \choose 1} } 
w_T \lp \bs{V} \rp, 
\) 
re-defining $\| \cdot \|$ and $h= |\FF|^g \| \HH \|^{-g}$ accordingly.

This just considers $u=1$ instead of all $u \in [r, m] \cap \Z$. The following statements can be verified similarly to Steps 2 and 3. 
\bdash 
\item From $\Pi_j$-ii)-b), 
\[
\sum_{ 
u \in [m] \atop 
S \in {Z_j \choose u} 
}
\frac{|\FF_i[S]|^g}{b^{-(g-1)u} {m_* \choose u}}
< |\FF_i|^g, 
\quad \Rightarrow \quad 
\sum_{
T \in {Z \choose 1} 
} 
\gamma_T < 1. 
\] 
\item 
\[
\sum_{s \in [q_*] \atop S \in {X_* \choose s}} 
\frac{\| \HH[S] \|_T^g}{b^{-(g-1)s}~{m_* \choose s} }
\le \gamma_T |\FF_i|^g,
\quad 
\textrm{for each~} T \in {Z \choose 1}. 
\]

\item $h>1$, and $w_{S, T}=\| \HH[S] \|_T$ for all $S \subset X_*$ and $T \in {Z \choose 1}$. 

\item $\HH$ satisfies the $\Gamma_g \lp b_\dagger, h \rp$-condition on $\| \cdot \|$. 

\item 
\[
\sum_{
T \in {Z \choose 1} } 
\frac{|\FF_{Y, i}[T]|^g  
}
{
b^{-g+1} m_*
}
< \frac{3}{\ep}  |\FF_{Y, i}|^g, 
\] 
for more than $1-\ep$ of all $Y \in {X_* \choose n_*/4}$, meaning 
\beeq{eqStep4} 
|\FF_{Y, i}[T]| < \frac{m^{2 \ep}}{b_{j+1}} |\FF_{Y, i}|, \quad 
\textrm{for every~} T \in {Z \choose 1}. 
\eeq 
\edash 

\medskip 

Eliminate all $Y$ with $\neg$ \refeq{eqStep4} from $\YY_i$.

\medskip 

\noindent
{\bf Step 5.~}{\em Update $\FF_i$ so they satisfy $\Pi_{j+1}$-i) and iii).} 
The obtained $\YY_i$ is a subfamily of ${X_* \choose n_*/4}$ such that $|\YY_i| > {n_* \choose n_*/4} (1-3\ep)$, \refeq{eqStep1}, and \refeq{eqStep3} $\wedge$ \refeq{eqStep4} for every $Y \in \YY_i$. 
Extend the $n_*/4$-sets in $\YY_i$ to $n'$-sets in the universal set $X_j$ where $n' := 3 q+ |X_j|/4$. 
Noting Remarks L) and M) of Section 2 with $q_* + n_*/4 < n' < |X_j|/3$, we see the obtained family $\YY'_i:= Ext \lp \YY_i, n' \rp$ has a cardinality at least $\lc {|X_j| \choose n'} \lp 1 - 3 \ep \rp \rc$. If $q_*=0$, the same goes for $\YY'_i := {X_j \choose n'}$. 

Performing the above for the three $i$, we have 
\[
\left \lceil {|X_j| \choose n'}
{|X_j|-n' \choose n'}
{|X_j|-2n' \choose n'}
\lp 1- 9 \ep \rp \right \rceil
\]
tripes $\lp Y_1, Y_2, Y_3 \rp$ such that $Y_i \in \YY'_i$, and $Y_i$ are mutually disjoint.  Fix such a triple $(Y_1, Y_2, Y_3)$.

Choose any $Y \in \YY_i \cap {Y_i \choose n_*/4}$ for each $i$, and update $\FF_i$ by $\FF_i \leftarrow \FF_{Y, i} $. By construction so far, the new $\FF_i$ satisfy: 
\be{A)} \setcounter{enumi}{8}
\item $\Pi_{j+1}$-i) and iii). 
\item If $j \le r$, 
\[ 
\sum_{ 
r \le u \le m \atop
S \in {Z \choose u} 
}
\frac{|\FF_i[S]|^g}{b_{j+1}^{-(g-1)u} {m_* \choose u}}
< \ep  |\FF_i|^g, 
\] 
\item and $|\FF_i[T]| < \frac{m^{2\ep}}{b_{j+1}}|\FF_i|$ for each $T \in {Z \choose 1}$. 
\qed 
\ee

\medskip

\noindent
{\bf Step 6.~}{\em Find a set $S \subset Z$ meeting some desired conditions.}  Reset $b$ by $b \leftarrow b_{j+1}$. Let $S$ denote a subset of $Z$. Find the maximum $v \in [0, m] \cap \Z$ such that 
\beeq{eqStep6}
\sum_{ 
S \in {Z \choose v} 
}
\frac{|\FF_i[S]|^g}{b^{-(g-1)v} {m_* \choose v}}
\ge |\FF_i|^g. 
\eeq 
There does exist such a $v$ less than $r$, as \refeq{eqStep6} is true for $v=0$ and false for $v \ge r$ by J).

Below we show the existence of $S \in {Z \choose v}$ satisfying the three conditions. 
\( 
1) && 
\sum_{r \le u \le m \atop 
T \in {Z-S \choose u} 
}
\frac{|\FF_i[S \cup T]|^g}{b^{-(g-1)u} {m_*-v \choose u}}
< \frac{1}{2} |\FF_i[S]|^g. 
\\ 2) && 
\sum_{1 \le u < r \atop 
T \in {Z-S \choose u} 
}
\frac{|\FF_i[S \cup T]|^g}{\lp \frac{b}{2} \rp^{-(g-1)u} {m_*-v \choose u}}
< |\FF_i[S]|^g. 
\\ 3) && 
|\FF_i[S]| > \frac{1}{2 b^v} |\FF_i|. 
\)
Assume $v>0$, otherwise these are clearly true by J).

Observe here that 
\[ 
\sum_{S \in {Z \choose v} \atop \textrm{with~} \neg 3)} 
\frac{|\FF_i[S]|^g}{b^{-(g-1)v} {m \choose v}}
\le 2^{-g+1} |\FF_i|^g, 
\]
similarly to having found $\Pi_1$-ii)-a) before Step 1. 
Therefore, if there were no $S \in {Z \choose v}$ such that 1) $\wedge$ 2) $\wedge$ 3), one of the following would be true: 
\(
- && 
\sum_{S \in {Z \choose v} \atop \textrm{with~} \neg 1)}  
\frac{|\FF_i[S]|^g}{b^{-(g-1)v} {m_* \choose v}}
\ge  \frac{1}{3} |\FF_i|^g; 
\\ - && 
\sum_{S \in {Z \choose v} \atop \textrm{with~} \neg 2)\textrm{-}u}  
\frac{|\FF_i[S]|^g}{b^{-(g-1)v} {m_* \choose v}}
\ge  \frac{1}{3^u} |\FF_i|^g, 
\quad 
\textrm{for some~} u \in [r-1], 
\)
where 2)-$u$ means 
\[
\sum_{
T \in {Z-S \choose u} 
}
\frac{|\FF_i[S \cup T]|^g}{\lp \frac{b}{2} \rp^{-(g-1)u} {m_*-v \choose u}}
< \frac{1}{2^u} |\FF_i[S]|^g. 
\] 
Call the two cases {\em Cases 1} and {\em 2}, respectively. 

\medskip 

We show a contradiction in Case 1 from 
\beeqn  
&& \label{eq2Step6} 
\sum_{S \in {Z \choose v} \atop \textrm{with~} \neg 1)} 
\sum_{ r \le u \le m \atop 
T \in {Z-S \choose u} 
}
\frac{|\FF_i[S \cup T]|^g}{b^{-(g-1)(v+u)} {m_* \choose v+u} {v + u \choose v}} 
\nexteqline \nonumber 
\sum_{S \in {Z \choose v} \atop \textrm{with~} \neg 1)} 
\frac{1}{b^{-(g-1)v} {m_* \choose v}}
\sum_{ r \le u \le m \atop 
T \in {Z-S \choose u} 
}
\frac{|\FF_i[S \cup T]|^g}{b^{-(g-1)u} {m_*-v \choose u}} 
\\ &\ge& \nonumber 
\sum_{S \in {Z \choose v} \atop \textrm{with~} \neg 1)} 
\frac{|\FF_i[S]|^g}{2 b^{-(g-1)v} {m_* \choose v}} 
\\ &\ge& \nonumber 
\frac{1}{6} |\FF_i|^g.  
\eeqn 
The inequality means 
\beeq{eq3Step6} 
\sum_{
r \le v' \le m \atop 
S' \in {Z \choose v'}} 
\frac{|\FF_i[S']|^g}{b^{-(g-1)v'} {m_* \choose v'}}  
\ge \frac{1}{6} |\FF_i|^g.  
\eeq 
See it as follows. Let 
\[
r_{v'} =  |\FF_i|^{-g}  \sum_{S' \in {Z \choose v'}} \frac{|\FF_i[S']|^g}{b^{-(g-1) v'} {m_* \choose v'}}, 
\]
for each $v' \in [r, m] \cap \Z$. 
It satisfies 
\[
\sum_{S \in {Z \choose v},~T \in {Z-S \choose v'- v}} 
|\FF_i[S \cup T]|^g 
\le 
\sum_{S' \in {Z \choose v'}} |\FF_i[S']|^g  {v' \choose v}
= 
r_{v'} |\FF_i|^g b^{-(g-1) v'} {m_* \choose v'} {v' \choose v}
,  
\]
since there are at most ${v' \choose v}$ pairs $(S, T)$ such that $S \cup T$ equals each given $S' \in {Z \choose v'}$. 
Summing up 
\[
\sum_{S \in {Z \choose v},~T \in {Z-S \choose v'- v}} 
\frac{|\FF_i[S \cup T]|^g }{b^{-(g-1)v'} {m_* \choose v'}{v' \choose v}}  
\le r_{v'} |\FF_i|^g 
. 
\]
for all $v'$, we see $\neg$ \refeq{eq3Step6} $\Rightarrow$ 
$\neg$ \refeq{eq2Step6}. Hence \refeq{eq3Step6}. This contradicts J), so Case 1 is impossible to occur.

Given $u \in [r-1]$ in Case 2, similarly find 
\(
&& 
\sum_{S \in {Z \choose v} \atop \textrm{with~} \neg 2)\textrm{-}u} 
\sum_{ 
T \in {Z-S \choose u} 
}
\frac{|\FF_i[S \cup T]|^g}{b^{-(g-1)(v+u)} {m_* \choose v+u} {v + u \choose v}} 
\\ &\ge&  
2^{(g-2)u}
\sum_{S \in {Z \choose v} \atop \textrm{with~} \neg 2)\textrm{-}u} 
\frac{|\FF_i[S]|^g}{b^{-(g-1) v} {m_* \choose v}} 
> \lp \frac{2^{g-2}}{3} \rp^u |\FF_i|^g, 
\)
meaning 
\[
\sum_{ 
S' \in {Z \choose v + u}} 
\frac{|\FF_i[S']|^g}{b^{-(g-1) (v+u)} {m_* \choose v+u}}  
> |\FF_i|^g.  
\]
It is against the maximality of $v$ such that \refeq{eqStep6}. 

For the $\FF_i$ updated by Step 5, we have proven the existence of $S \subset Z$ such that $|S| < r$ and 1) $\wedge$ 2) $\wedge$ 3). Denote it by $S_i$.

\medskip

\noindent
{\bf Step 7.~}{\em Perform the final updates on $C_i$ and $\FF_i$ for $\Pi_{j+1}$.}  
Let $i=1$. 
With the obtained $S_i$, update $C_i$ by $C_i \leftarrow C_i \cup S_i$  and $\FF_i$ by $\FF_i \leftarrow \FF_i [C_i]$. Then let the other two $\FF_{i'}$ $\lp i' \in [3] - \lb i \rb \rp$ exclude $S_i$ as we did before in the preprocess with $
\FF_{i'} \leftarrow \FF_{i'} \cap {X - S_i \choose m} 
$. 
By K) and $|S_i|<r$, this could reduce $|\FF_{i'}|$ only by a factor greater than $1-m^{-\delta/2}$, thus $|\FF_{i'}|^g$ by one greater than $1 -\ep$, affecting the subsequent updates trivially. We note here that the same result at the end of Step 6 holds even if $\ep$ in J) is replaced by $2 \ep$. 

Set $i=2$ to perform the same construction, where $\FF_1$ can exclude $S_2$ by 2) of Step 6 since it means the $\Gamma \lp 2 b m^{-\ep}  \rp$-condition of $\FF_1$ in $Z$.  

Finally set $i=3$ to perform the same construction on $\FF_i$ with the other $\FF_{i'}$ excluding $S_3$. Then the three new $C_i$ and $\FF_i$ all satisfy $\Pi_{j+1}$: if $j \le r$, the property $\Pi_{j+1}$-i) $\wedge$ iii) is true by I), $|S_i| < r$ and 3) of Step 6. Also ii)-a) holds by 1), and ii)-b) by 2).

\medskip

This completes our updates proving the induction step $\Pi_j \Rightarrow \Pi_{j+1}$.  We now have \refth{3SF}.

\medskip

\appendix
\section*{Appendix 1: Proof of \reflm{PhaseII}} 
\setcounter{section}{1}

\medskip

\noindent {\bf Lemma \ref{PhaseII}.}
{\em For $\FF \subset{X \choose m}$ such that $m \le \frac{n}{2}$,}
\[
\kappa \lbr {X \choose 2m} - Ext \lp \FF, 2m \rp \rbr
\ge
2 \kappa \lbr {X \choose m} - \FF \rbr.
\]
\begin{proof}
For each $S \in {X \choose m}-\FF$ and $j \in [0, m] \cap \Z$, let
\[
\FF_j = \lb T - S ~:~ T\in \FF,~ |T - S| = j \rb.
\]
There exists $j$ such that $\kappa \lp \FF_j \rp$ in the universal set $X- S$ is at most $\kappa \lp \FF \rp$ in X, otherwise
\[
|\FF| < \sum_{j\ge 0} {m \choose m-j}{n-m \choose j} e^{-\kappa \lp \FF \rp} = {n \choose m} e^{-\kappa \lp \FF \rp} = |\FF|.
\]
Taking $Ext \lp \FF_j, m \rp$ in $X-S$ with Remark A) of Sec.\ 2.3, we see
there are $\lc {n -m \choose m} e^{-\kappa \lp \FF \rp} \rc$ pairs $(S, U)$ such that $U \in {X-S \choose m}$ and $S \cup U \in Ext \lp \FF, 2m \rp$ for each $S \in {X \choose m}- \FF$.

Now consider all pairs $(S, U)$ such that $S$ and $U$ are disjoint $m$-sets, and $S \cup U \in Ext \lp \FF, 2m \rp$. Their total number is at least
${n \choose m}{n-m \choose m}={n \choose 2m}{2m \choose m}$
times
$
(1-z) + z (1-z) =1 - z^2
$
where $z=e^{-\kappa \lbr {X \choose m} - \FF \rbr}$.

As a $2m$-set produces at most ${2m \choose m}$ pairs $(S, U)$, there are at least $\lc {n \choose 2m}(1-z^2) \rc$ sets in $Ext \lp \FF, 2m \rp$.
The lemma follows.
\end{proof}

\section*{Appendix 2: Asymptotics of Binomial Coefficients} 
Let
\beeq{FunctionS}
s: (0, 1) \rightarrow (0,1),
\hspace*{5mm}
t \mapsto 1 - \lp 1 - \frac{1}{t} \rp \ln (1-t).
\eeq
By the Taylor series of $\ln(1-t)$, the function is also expressed as
\beeq{ExpandS}
s(t)
=
1 +\lp 1 - \frac{1}{t} \rp \sum_{j \ge 1} \frac{t^j}{j}
= 1 + \sum_{j \ge 1} \frac{t^j}{j} - \sum_{j\ge 0} \frac{t^j}{j+1}
= \sum_{j \ge 1} \frac{t^j}{j(j+1)}.
\eeq
We have the following double inequality.

\begin{lemma} \label{asymptotic}
For $x, y \in \Z_{>0}$ such that $x<y$,
\(
&&
\frac{1}{12 x+1} - \frac{1}{12 y} - \frac{1}{12(x-y)}
<
z
<
\frac{1}{12 x} - \frac{1}{12 y+1} - \frac{1}{12(x-y)+1},
\\ \textrm{where} &&
z = \ln {x \choose y} - y \lbr \ln \frac{x}{y} + 1 - s \lp \frac{y}{x} \rp  \rbr
- \frac{1}{2} \ln \frac{x}{2 \pi y(x-y)}.  \qed
\)
\end{lemma}
\begin{proof}
Stirling's approximation in form of double inequality is known as
\[
\sqrt{2 \pi n} \lp \frac{n}{e} \rp^n
\exp \lp \frac{1}{12n+1} \rp
< n! <
\sqrt{2 \pi n} \lp \frac{n}{e} \rp^n
\exp \lp \frac{1}{12n} \rp,
\]
for any $n \in \Z_{>0}$ \cite{R55}. By this we find
\(
&&
\frac{1}{12 x+1} - \frac{1}{12 y} - \frac{1}{12(x-y)}
<
\ln \frac{{x \choose y}}{u}
<
\frac{1}{12 x} - \frac{1}{12 y+1} - \frac{1}{12(x-y)+1},
\\&&
\eqwhere
u
= \sqrt{\frac{x}{2 \pi y(x-y)}} \frac{x^x}{y^y (x-y)^{x-y}}.
\)
Since
$
\ln \frac{x^x}{y^y (x-y)^{x-y}}  = y \lbr \ln \frac{x}{y} + 1 - s \lp \frac{y}{x} \rp  \rbr
$
by \refeq{FunctionS}, it proves the lemma.
\end{proof}

The lemma is useful to approximate $\ln {x \choose y}$ by $z$ to an error less than $1/6$. It derives the standard estimate of the binomial coefficient, $i.e.$, 
$
{x \choose y} < \lp \frac{e x}{y} \rp^y
$, due to $- y s \lp \frac{y}{x} \rp + \frac{1}{10x}<0$ from \refeq{ExpandS} and $\frac{x}{2 \pi y (x-y)} <1$.  So 
\[
\lp \frac{x}{y} \rp^y \le {x \choose y} < \lp \frac{e x}{y} \rp^y, 
\]
for every $x, y \in \Z_{>0}$ with $x \ge y$, as ${x \choose y} = \prod_{j=0}^{y-1} \frac{x-j}{y-j} \ge \lp \frac{x}{y} \rp^y$. 

We also have: 

\begin{lemma} \label{SpaceProp} 
For $x \in \Z_{>1}$, $y \in (1, x) \cap \Z$, $x' \in [1, x/2] \cap \Z$, and $y' \in [y-1]$ with $y' > x' + y- x$, 
\(
&&
\ln {x \choose y} - \ln {x-x' \choose y-y'}{x' \choose y'}
> 
\frac{7 \lp y' - \frac{x' y}{x}  \rp^2}{8 y}
+ \frac{-1 + \ln 2 \pi z}{2}, 
\\ \textrm{where} && 
z = \lp 1- \frac{1}{y} \rp \lp 1- \frac{y- y'}{x- x'} \rp \lp 1- \frac{y'}{x'} \rp 
\lp 1 - \frac{y}{x} \rp^{-1}.  
\) 
\end{lemma}
\begin{proof}
Apply \reflm{asymptotic} to the three binomial coefficients to see 
\[
\ln {x-x' \choose y-y'}{x' \choose y'} - \ln {x \choose y}
< u - \sum_{j \ge 1} \frac{v_j}{j(j+1)} 
+ \frac{1 + \ln \frac{w}{2 \pi}}{2},
\]
where
\(
&&
u =
\lp y-  y' \rp \lp \ln \frac{x-x'}{y-y'} + 1\rp
+ y' \lp \ln \frac{x'}{y'} + 1 \rp
- y \lp \ln \frac{x}{y} + 1 \rp,
\\  &&
v_j =\frac{(y-y')^{j+1}}{(x-x')^j} + \frac{y'^{~j+1}}{x'^{~j}}
- \frac{y^{j+1}}{x^j}, 
\\ \textrm{and} &&
w = \frac{x-x'}{(y-y')(x - x'-y + y')} \cdot  
\frac{x'}{y'(x' - y')} \cdot  
\frac{y(x - y)}{x}. 
\) 
As $w \le \frac{1}{z}$ from  
$ 
\frac{y}{y'(y-y')} \le  \lp 1- \frac{1}{y} \rp^{-1} 
$, it suffices to show 
\beeq{eq2SpaceProp} 
v_j \ge 0, \quad \textrm{for all~} j \in \Z_{>0}, 
\eeq
and 
\beeq{eqSpaceProp}
u< -\frac{7 \Delta^2}{8 y},
\eqwhere
\Delta = y' - \frac{x' y}{x}. 
\eeq

To see \refeq{eq2SpaceProp}, put 
\[
a = \frac{x'}{x}, \quad 
t =\frac{y'}{y}, \eqand 
f =
\frac{\lp 1- t \rp^{j+1}}{\lp 1- a \rp^j}
+ \frac{t^{j+1}}{a^j}, 
\]
for a given $j$. 
Then the desired condition holds if $f \ge 1$ for each fixed $a \in (0, 1)$ and all $t \in \R_{\ge 0}$. It is straightforward to check its truth.

We show \refeq{eqSpaceProp} finding that 
\(
u &=&
y' \ln \frac{x'y}{xy'} + \lp y - y' \rp \ln \frac{1- \frac{x'}{x}}{1-\frac{y'}{y}} 
\\ &=&
- \Delta \lp \frac{1}{p} + 1 \rp \ln \lp 1 + p \rp
- \Delta \lp \frac{1}{q} -1 \rp  \ln \lp 1 -  q \rp, 
\\ \textrm{where} && 
p = \frac{x \Delta}{x'y}, \eqand
q= \frac{\Delta}{y \lp 1-\frac{x'}{x} \rp}. 
\)

Observe facts.

\bdash 
\item $|q| < 1 \Leftrightarrow \frac{2x'}{x} -1 < \frac{y'}{y} < 1$ is true by $x' \le x/2$.
\item By the Taylor series of the natural logarithm,
\(
&& \lp \frac{1}{p} + 1 \rp \ln \lp 1 + p \rp
= 1 - \sum_{j \ge 1} \frac{(-p)^j}{j(j+1)},
\eqspace \textrm{if $|p|<1$,}
\\ \textrm{and} &&
\lp \frac{1}{q} -1 \rp  \ln \lp 1 -  q \rp
= - 1 + \sum_{j \ge 1} \frac{q^j}{j(j+1)},
\eqspace \textrm{as $|q|<1$,}
\)

\item So
\[
u = \Delta \sum_{j \ge 1} \frac{(-p)^j}{j(j+1)} - \Delta \sum_{j \ge 1} \frac{q^j}{j(j+1)} 
< \Delta \lp \frac{-p}{2} + \frac{p^2}{6} - \frac{q}{2} \rp
< - \frac{\Delta^2}{y}, 
\]
if $|p|<1$. 

\item $\lp 1 + \frac{1}{p} \rp \ln \lp 1 + p \rp \ge 2 \ln 2$ if $p \ge 1$, so 
\[
u < \lp -2 \ln 2 + 1  \rp \Delta - \frac{\Delta^2}{2 y}< - \frac{7 \Delta^2}{8 y}, 
\]
since $\frac{x' p}{x} = \frac{\Delta}{y}  = \frac{y'}{y} - \frac{x'}{x} \in (0, 1)$. 
\edash

\medskip

Hence \refeq{eqSpaceProp}, completing the proof. 
\end{proof}

\bibliographystyle{amsplain}


\end{document}